\documentclass[11pt,xcolor=dvipsnames]{article}
\pdfoutput=1 
\usepackage[utf8x]{inputenc}
\usepackage[english]{babel}
\usepackage{graphicx}           
\usepackage{url}
\usepackage{eurosym}
\usepackage{makeidx}
\usepackage[plain]{algorithm}
\usepackage{algorithmic} 
\usepackage{color} 
\usepackage[margin=3.3cm]{geometry}
\usepackage{setspace}
\usepackage{datetime}
\usepackage{pifont}
\usepackage{tikz}
\usepackage{natbib}
\usepackage{hyperref}
\usepackage{amsmath, amsthm,amsfonts,amssymb,latexsym,stmaryrd}

\renewcommand{\L}{\mathcal{L}}		% générateur infinitésimal
\newcommand{\F}{\mathcal{F}}			% filtration
\newcommand{\C}{\mathcal{C}}			% espace des traits
\newcommand{\D}{\mathcal{D}}	
\newcommand{\M}{\mathcal{M}}			% espaces de mesures
\newcommand{\EE}{\mathbb{E}}			% espérance
\newcommand{\PP}{\mathbb{P}}			% probabilité
\newcommand{\RR}{\mathbb{R}}			% ensemble des nombres rééls
\newcommand{\NN}{\mathbb{N}}			% ensembles des entiers naturels

\newcommand{\taudiv}{b}							% taux de division
\newcommand{\mut}{\gamma}						% proba de mutation
\newcommand{\mmax}{M}    						% masse maximale d'un individu
\newcommand{\mdiv}{m_{\textrm{\tiny\rm div}}}	% masse minimale de division
\newcommand{\rhog}{g}							% vitesse de croissance
\newcommand{\Sin}{{\mathbf s}_{\textrm{\tiny\rm in}}} %concentr. en substrat en entrée

\newcommand{\MP}{\mathcal{N}}			% mesure aléatoires de poisson
\newcommand{\op}{\mathcal{G}}			% opérateur du modéle IDE réduit
\newcommand{\densiteIDE}{r}				% densité de population du modèle IDE
\newcommand{\densiteIDEr}{m}				% densité de population du modèle IDE réduit

\newcommand{\var}{{\textrm{Var}}}								% variance
\newcommand{\eqdef}     {\stackrel{{\textrm{\rm\tiny def}}}{=}}	% définition
\newcommand{\indic}{{{\mathrm\mathbf1}}}							% indicatrice
\newcommand{\umoins}{{\smash{u^{\raisebox{-1pt}{\scriptsize\scalebox{0.5}{$-$}}}}}}

\newcommand{\dif}{{{\textrm{\upshape d}}}}

\newcommand{\norme}[1]{\left\Vert #1 \right\Vert }
\newcommand{\crochet}[1] {\langle #1 \rangle}

\renewcommand{\tilde}{\widetilde}

\newtheorem{theorem}      {Theorem}[section]
\newtheorem{theorem*}     {theorem}
\newtheorem{proposition}  [theorem]{Proposition}

\newtheorem{lemma}        [theorem]{Lemma}

\newtheorem{remark}       [theorem]{Remark}

\newtheorem{corollary}    [theorem]{Corollary}

\newtheorem{hypothesis}   [theorem]{Assumption}
\newtheorem{hypotheses}   [theorem]{Assumptions}

\begin{document}
%%%%%%%%%%%%%%%%%%%%%%%%%%%%%%%%%%%%%%%%%%%%%%%%%%%%%%%%%%%%%%%%%%%%%%
%%%%%%%%%%%%%%%%%%%%%%%%%%%%%%%%%%%%%%%%%%%%%%%%%%%%%%%%%%%%%%%%%%%%%%
\title{
Links between deterministic and stochastic approaches for invasion in growth-fragmentation-death models}

\author{Fabien Campillo$^{1,2}$ \and Nicolas Champagnat$^{3,4,5}$ \and Coralie Fritsch$^{3,4,5}$}

\footnotetext[1]{Inria, LEMON, Montpellier, F-34095, France}

\footnotetext[2]{Institut Montpelli\'erain Alexander Grothendieck, Montpellier, F-34095, France}

\footnotetext[3]{Universit\'e de Lorraine, Institut Elie Cartan de Lorraine,
    UMR 7502, Vand\oe uvre-l\`es-Nancy, F-54506, France}

\footnotetext[4]{CNRS, Institut Elie Cartan de Lorraine, UMR
    7502, Vand\oe uvre-l\`es-Nancy, F-54506, France}

\footnotetext[5]{Inria, TOSCA, Villers-l\`es-Nancy, F-54600, France \protect \\ 
				E-mail: {fabien.campillo@inria.fr}, {nicolas.champagnat@inria.fr},
				 {coralie.fritsch@inria.fr}}

%%%%%%%%%%%%%%%%%%%%%%%%%%%%%%%%%%%%%%%%%%%%%%%%%%%%%%%%%%%%%%%%%%%%%%
%%%%%%%%%%%%%%%%%%%%%%%%%%%%%%%%%%%%%%%%%%%%%%%%%%%%%%%%%%%%%%%%%%%%%%

%%%%%%%%%%%%%%%%%%%%%%%%%%%%%%%%%%%%%%%%%%%%%%%%%%%%%%%%%%%%%%%%%%%%%%
\maketitle
%%%%%%%%%%%%%%%%%%%%%%%%%%%%%%%%%%%%%%%%%%%%%%%%%%%%%%%%%%%%%%%%%%%%%%

%%%%%%%%%%%%%%%%%%%%%%%%%%%%%%%%%%%%%%%%%%%%%%%%%%%%%%%%%%%%%%%%%%%%%%
\begin{abstract}
We present two approaches to study invasion in growth-fragmentation-death models. The first one is based on a stochastic individual
  based model, which is a piecewise deterministic branching process with a continuum of types, and the second one is based on an
  integro-differential model. The invasion of the population is described by the survival probability for the former model and by an
  eigenproblem for the latter one. We study these two notions of invasion fitness, giving different characterizations of the growth
  of the population, and we make links between these two complementary points of view. In particular we prove that the two approaches
  lead to the same criterion of possible invasion. Based on Krein-Rutman theory, we also give a proof of the existence of a solution
  to the eigenproblem, which satisfies the conditions needed for our study of the stochastic model, hence providing a set of
  assumptions under which both approaches can be carried out. Finally, we motivate our work in the context of adaptive dynamics in a
  chemostat model.
  

\paragraph{Mathematics Subject Classification (MSC2010):}
60J80, 60J85, 60J25, 35Q92, 45C05, 92D25.
\end{abstract}
%%%%%%%%%%%%%%%%%%%%%%%%%%%%%%%%%%%%%%%%%%%%%%%%%%%%%%%%%%%%%%%%%%%%%%

\section{Introduction}
\label{intro}

A chemostat is an experimental biological device allowing to grow micro-organisms in a controlled environment. It was first developed by \cite{monod1950a} and \cite{novick1950a} and is now widely used in biology and industry, for example for wastewater
treatment \citep{Bengtsson2008a}.
The chemostat is a continuous culture method for maintaining a bacterial ecosystem growing in a substrate. Bacteria are cultivated in a container with a supply of substrate and a biomass and environment withdrawal so that the volume within the container is kept constant. An important aspect still little discussed in modeling terms, is the evolution of bacteria that adapt to the conditions maintained by the practitioner in the chemostat \citep{Wick2002a, kuenen2009a}.

The theory of Adaptive Dynamics proposes a set of mathematical tools to analyze the adaptation of biological populations \citep{Metz1996a, dieckmann1996a, geritz1998a}. These tools are based on the description of evolution as an adaptive walk
\citep{gillespie1983a} and rely heavily on the notion of \emph{invasion fitness} \citep{Metz1992a}. The invasion fitness describes the
ability of a mutant strain to invade a stable resident population, and its mathematical characterization requires to analyze the
long-time growth or decrease of a mutant population in the fixed environment formed by the resident population. The goal of this work
is to study the invasion fitness in mass-structured models, where bacteria are characterized by their mass, which grow continuously
over time, and is split between offspring at division times. More generally, we will consider general growth-fragmentation-death
models, covering the case of bacteria in a chemostat.

The growth of a mutant population in a chemostat can be modeled either in a continuous or a discrete population. In the
mass-structured case, the former is typically represented as an integro-differential equation of growth-fragmentation \citep{fredrickson1967a, Doumic2010a}
and the latter as a stochastic individual-based model \citep{deangelis1992a, champagnat2008a} where individuals
are characterized by their mass \citep{campillo2014d, fritsch2014a, doumic2012a}. The latter is typically a
piecewise deterministic Markov process since the growth of individual masses is usually modeled deterministically and the birth and
death of individuals are stochastic. This class of models has received a lot of interest in the past years \citep{jacobsen2006a}.
Adaptive dynamics were already studied in deterministic or stochastic chemostat models without mass structure in \citep{Doebli2002, mirrahimi2012a, champagnat2014a}. Age-structured stochastic models were also studied by \cite{TranChi2009a}, but their results do not
extend to the case of mass-structure since the mass of the offsprings at a division time is random, while their age is
always 0 for age-structured models.

Deterministic and stochastic models are often considered equivalent in adaptive dynamics, but this is not always
the case. In particular, the notion of invasion fitness has a completely different biological meaning in each case  \citep{Metz1992a, metz2008a}. In a stochastic discrete model, it is defined as the probability of invasion (or survival) of a
branching process describing the mutant population initiated by a single mutant individual. In a deterministic model, it is defined
as the asymptotic rate of growth (or decrease) of the mutant population, characterized as the solution of an eigenproblem. The former
is relevant to study invasions of mutant populations which appear separately and with one initial individual in the population, whereas the latter is
meaningful when several mutant strains are present in a small density simultaneously, and, under the competitive exclusion principle, the successful one is the one which grows
faster.

The goal of the present work is to present the two points of view in a unified framework and to study the link between these two notions of
invasion fitness in the context of growth-fragmentation-death models. In particular, we prove that the two approaches lead to the same criterion of
possible invasion, although the link between the two notions of fitness is far from obvious. This means that one can choose
arbitrarily the point of view which is more suited to the particular problem or application under study. For example, the invasion
fitness is simpler to characterize in the deterministic model, and is hence more convenient to study qualitative evolutionary
properties of practical biological situations. It is also less costly to compute the deterministic invasion fitness using classical numerical
approximation methods than to compute the extinction probability using Monte-Carlo or iterative methods \citep{fritsch2015b}. In addition, the criterion of invasion is more straightforward in the deterministic case (sign of the eigenvalue) than in the
stochastic one (probability of extinction equal to or different from 1). Conversely, for modeling purpose, the stochastic model is more
convenient in small populations, which is typically the case in invasion problems. The invasion fitness also provides further
information on the stochastic model, since the extinction probability depends of the initial state of the population.

The two approaches are also complementary from the mathematical point of view. For example, the existence of a solution to the
eigenvalue problem facilitates the study of the individual-based model since it provides natural martingales (see Lemma \ref{lemme.martingale.bornee} below).
Conversely, the stochastic process gives a new point of view to PDE analysis questions, like the maximum principle (see Corollary \ref{da.vitesse.cv}).
More generally, it allows to give a description of the population generation by generation, and can help to prove results that would
not be obvious with a purely deterministic approach. An example is given by monotonicity properties of the invasion fitnesses with
respect to the substrate concentration in a chemostat, as shown by \citep{campillo2015b}.

Our results also give the precise conditions that should satisfy the solution to the eigenvalue problem to make the link between the
two notions of invasion fitness. 
The results of \cite{Doumic2007a} and \cite{Doumic2010a} about existence of eigenelements do not cover some biological cases as bacterial Gompertz growth, so we give a full proof of this result based on Krein-Rutman theory, hence providing a set of assumptions under
which both approaches can be carried out.
%{\color{red}It turns out that these assumptions do not fit the former results of \cite{Doumic2007a} and
%\cite{Doumic2010a}, so we give a full proof of this result based on Krein-Rutman theory, hence providing a set of assumptions under
%which both approaches can be carried out.}
We conclude by emphasizing that, although general results making the link between the
survival probability in multitype branching processes and the dominant eigenvalue of an appropriate matrix are well-known, our study
is non-trivial because the piecewise-deterministic stochastic process is an infinite dimensional branching process (actually with
a continuum of types) and the eigenvalue problem itself is non-trivial. In particular, a key technical step in our work is Proposition \ref{majoration.esp.Nt} giving a control of the expected total population size.

% complementary: the study of one model give non-trivial information on the second one. For exemple, using the individual-based model,
% we can prove properties of monotonicity with respect to an environmental parameter of the eigenvalue associated to the deterministic
% model, which seems not easy to prove using a purely deterministic approach ({\color{red}citer 2nd papier}). More generally, the
% stochastic model allows to study the population through generations and to define an invasion fitness depending on the initial state
% of the population, which is not the case for the deterministic model. Conversely, the deterministic model allows a simpler
% characterization of invasion fitness, which is more convenient for the study of practical biological situations and of qualitative
% evolutionary properties (direction of evolution, existence of an ESS...). The deterministic model is not appropriate in some cases as
% in small population size or invasion problem, however, through the link between the two invasion fitness, it can be used in order to
% greatly facilitate invasion problem of the stochastic model, which can be complex to numerically analyse ({\color{red}Fritsch
%   Campillo Ovaskainen}).

%section Model
Section \ref{modele} is devoted to the description of the deterministic and stochastic models that we study. 
%section Extinction Probability
Section \ref{sec.extinction.proba} is devoted to the survival probability in the stochastic individual-based model. In particular, we characterize this probability as
the solution of an integral equation and we prove that a positive probability of survival for one initial mass implies the same
property for all initial masses. 
%section links between the 2 approaches
In Section \ref{subsec.stoch.proc.eigenpbl}, we give a stochastic representation of the solution of the deterministic problem and deduce some bounds on these solutions. The precise link between the invasion criteria of the two models is given in Theorem \ref{th.critere} of Subsection \ref{subsec.invasion.fitness}. 
%section eigenproblem
In Section \ref{sec.gfd.eigenpbl}, we give a set of assumptions
under which we can prove the existence of a solution to the eigenvalue problem satisfying the conditions of our main results of Section \ref{sec.lien.approches}. 
%section application to a chemostat model
Section \ref{sec.chemostat} is devoted to a more detailed discussion of the motivation of our work in the context of adaptive dynamics in chemostat models.

%%%%%%%%%%%%%%%%%%%%%%%%%%%%%%%%%%%%%%%%%%%%%%%%%%%%%%%%%%%%%%%%%%%%%%
%%%%%%%%%%%%%%%%%%%%%%%%%%%%%%%%%%%%%%%%%%%%%%%%%%%%%%%%%%%%%%%%%%%%%%
\section{Modeling of the problem}
\label{modele}
%%%%%%%%%%%%%%%%%%%%%%%%%%%%%%%%%%%%%%%%%%%%%%%%%%%%%%%%%%%%%%%%%%%%%%
%%%%%%%%%%%%%%%%%%%%%%%%%%%%%%%%%%%%%%%%%%%%%%%%%%%%%%%%%%%%%%%%%%%%%%

In this section, we present two descriptions of the growth-fragmentation-death model determined by the mechanisms of the Section \ref{subsec.mecha}. The first one, presented in Section \ref{subsec.edp}, is deterministic whereas the second one, presented in Section \ref{subsec.IBM}, is stochastic. The determinitic description is valid only in large population size whereas the stochastic one is valid all the time but is impossible to simulate, in an exact way, in large population size.

\subsection{Basic mechanisms}
\label{subsec.mecha}

We consider a growth-fragmentation-death model in which one individual is characterized by its mass $x \in [0,\mmax]$ where $\mmax$ is the maximal mass of individuals and is affected by the following mechanisms:

\begin{enumerate}
\item  \textbf{Division:} each individual, with mass $x$, divides at rate $\taudiv(x)$, into two individuals with masses $\alpha \, x$ and $(1-\alpha)\,x$, where the proportion $\alpha$ is distributed according to the kernel of probability density function $q(\alpha)$ on $[0,1]$.
\begin{center}
\includegraphics[width=5cm]{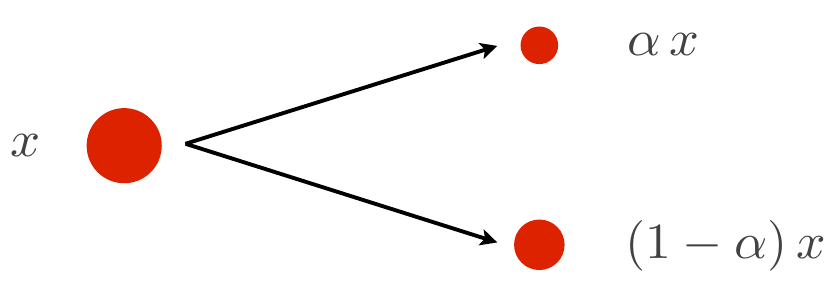}
\end{center}

\item \textbf{Death:} each individual dies at rate $D$.

\item \textbf{Growth:} between division and death times, the mass of an individual grows at speed $g:[0,\mmax]\to\RR_+$, i.e.
\begin{align*}
  \frac{\dif}{\dif t} x_t
  = 
  g(x_t)\,.
\end{align*}
\end{enumerate}

We set $A_t$ the growth flow defined for any $t\geq 0$ and $x\in[0,\mmax]$ by
\begin{align}
\label{def.flot}
  \frac{\dif}{\dif t}A_t(x) = g(A_t(x))\,,
  \qquad A_0(x)=x\,.
\end{align}

Throughout this paper we assume the following set of assumptions.

%----------------------------------------
\begin{hypotheses}
\label{da.hypo.model.reduit}
\begin{enumerate}
%------
\item The kernel $q$ is symmetric with respect to $1/2$: pour tout $\alpha \in [0,1]$,
$
	q(\alpha) = q(1-\alpha)\,.
$

%------
\item \label{g.0.M} $g(0)=g(M)=0$ and  $g(x)>0$ for any $x \in (0,M)$ .

%------
\item \label{hyp.g}$g\in C[0,\mmax]\cap C^1(0,\mmax)$

%------
\item $b\in C[0,\mmax]$ and there exists $\mdiv \in [0,M)$ and $\bar \taudiv>0$ such that
\begin{align*}
	\taudiv(x)=0 & \textrm{ if } x \leq \mdiv \,, 
\\
	0<\taudiv(x)\leq\bar \taudiv & \textrm{ if } x \in (\mdiv,M)\,.
\end{align*}
\end{enumerate}
\end{hypotheses}
%----------------------------------------

Assumption \ref{da.hypo.model.reduit}-\ref{hyp.g} ensures existence and uniqueness of the growth flow defined by \eqref{def.flot} for $x\in(0,\mmax)$ until the exit time $T_{\text{exit}}(x)$ of $(0,\mmax)$ which can be finite. According to Assumption \ref{da.hypo.model.reduit}-\ref{g.0.M}, the flow can be defined as constant in 0 and $\mmax$ (this flow is not necessary unique). Moreover, $A\in C^1(\D)$ with $\D=\{(t,x),\, t<T_{\text{exit}}(x) \}$ \cite[Th. 6.8.1]{demazure2000a}.

%%%%%%%%%%%%%%%%%%%%%%%%%%%%%%%%%%%%%%%%%%%%%%%%%%%%%%%%%%%%%%%%%%%%%%
\subsection{Growth-fragmentation-death integro-differential model}
\label{subsec.edp}
%%%%%%%%%%%%%%%%%%%%%%%%%%%%%%%%%%%%%%%%%%%%%%%%%%%%%%%%%%%%%%%%%%%%%%

The deterministic model associated to the previous mechanisms is given by the integro-differential equation
\begin{align*}
\frac{\partial}{\partial t} \densiteIDEr_t(x)
	&=
	\op m_t(x)
\end{align*}
where $\densiteIDEr_t(x)$ represents the density of individuals with mass $x$ at time $t$, with the initial condition $\densiteIDEr_0$ and $\op$ is defined  or any $f\in C^1(0,\mmax)$, $x\in(0,\mmax)$ by
\begin{align*}
  \op f(x)
  \eqdef
  - \partial_x (g(x)\,f(x))
  - (D+\taudiv(x))\, f(x)
  + 2\int_0^\mmax \frac{\taudiv(z)}{z}\, 
		q\left(\frac{x}{z}\right)\, 
		f(z)\,\dif z\,.
\end{align*}
The adjoint operator $\op^*$ of $\op$ is defined for any $f\in C^1(0,\mmax)$, $x\in(0,\mmax)$ by
\begin{align}
\label{def.G.star}
  \op^* f(x)
  \eqdef
  - (D+\taudiv(x))\, f(x)
  + g(x)\,\partial_x f(x)
  + 2\,\taudiv(x)\,\int_0^1q(\alpha)\,f(\alpha\, x)\,\dif \alpha\,.
\end{align}

We consider the eigenproblem
\begin{subequations}
\label{eq.eigenproblem}
\begin{align}
\op \hat u(x) = \Lambda \, \hat u(x)\,,
\end{align}
\begin{align}
   \lim_{x\to 0}g(x)\, \hat u(x) &= 0\,,
   & 
   D+\Lambda &>0\,,
   &
   \hat u(x) &\geq 0\,,
   & 
   \int_0^\mmax \hat u(x)\,\dif x &= 1
\end{align}
\end{subequations}
and the adjoint problem
\begin{subequations}
\label{eq.eigenproblem.dual}
\begin{align}
\label{eigenequation.dual}
\op^* \hat v(x) = \Lambda \, \hat v(x)\,,
\end{align}
\begin{align}
   \hat v(x) &\geq 0\,,
   & 
   \int_0^\mmax \hat v(x)\,\hat u(x)\,\dif x &= 1\,.
\end{align}
\end{subequations}

\medskip

The eigenvalue $\Lambda$ is then interpreted as the exponential growth rate (we will give an individual justification of this interpretation below). The sign of $\Lambda$ gives an explosion criterion of the population: if $\Lambda>0$ then the population goes to infinity ; if $\Lambda<0$ then the population goes to extinction ; if $\Lambda=0$ then the population cannot explode.

\medskip

Existence and uniqueness of a solution $(\hat u,\hat v, \Lambda)$ of the system \eqref{eq.eigenproblem}-\eqref{eq.eigenproblem.dual} were proved for models which are relatively close to ours. \citet{Doumic2007a} proved this result for a growth-fragmentation model which is structured by mass and age. \citet{Doumic2010a} studied a mass-structured growth-fragmentation model with unbounded individual masses.
Since none of the results cover our situation, we give a proof of the existence of a solution based on an adaptation of the method of \cite{Doumic2007a} in Section \ref{sec.gfd.eigenpbl} under particular assumptions.

%%%%%%%%%%%%%%%%%%%%%%%%%%%%%%%%%%%%%%%%%%%%%%%%%%%%%%%%%%%%%%%%%%%%%%
\subsection{Growth-fragmentation-death individual-based model}
\label{subsec.IBM}
%%%%%%%%%%%%%%%%%%%%%%%%%%%%%%%%%%%%%%%%%%%%%%%%%%%%%%%%%%%%%%%%%%%%%%

We now describe the mass-structured individual-based model associated to the mechanisms described in Section \ref{subsec.mecha}.

We represent the population at time $t$ by the counting measure
\begin{align}
\label{da.pop.modele.reduit}
	\eta_t(\dif x) \eqdef \sum_{i=1}^{N_t}\delta_{X_t^i}(\dif x)\,,
\end{align}
where $N_t=\crochet{\eta_t,1}$ is the number of individuals in the population at time $t$ and $(X_t^i,\, i=1,\dots, N_t)$ are the masses of the $N_t$ individuals (arbitrarily ordered).

We consider two independent Poisson random measures $\MP_1(\dif u, \dif j, \dif \alpha, \dif \theta)$ and $\MP_2(\dif u, \dif j)$ defined on $\RR_+ \times \NN^* \times [0,1] \times [0,1]$ and $\RR_+ \times \NN^*$ respectively, corresponding to the division and death mechanisms respectively, with respective intensity measures
\begin{align*}
n_1(\dif u, \dif j, \dif \alpha, \dif \theta)
	& =
		\bar \taudiv \, \dif u \, \Big(\sum_{\ell \geq 1} \delta_{\ell}(\dif j) \Big)
		\, q(\dif \alpha)\, \dif \alpha \, \dif \theta\,,
\\
n_2(\dif u, \dif j)
	& =
		D \, \dif u \, \Big(\sum_{\ell \geq 1} \delta_{\ell}(\dif j) \Big)\,.
\end{align*}

Suppose $\MP_1$, $\MP_{2}$ and $\eta_{0}$ mutually independent, let $(\F_t)_{t \geq 0}$ be the canonical filtration generated by $\eta_0$, $\MP_1$ and $\MP_2$.
The process $(\eta_t)_{t\geq 0}$ is then defined by
\begin{align}
\nonumber
  &\eta_{t}
  =
  \sum_{j=1}^{N_0}\delta_{A_t(X_0^j)}
  + \iiiint\limits_{[0,t]\times\NN^*\times[0,1]^2}
        1_{\{j\leq N_{\umoins}\}} \, 
        1_{\{\theta \leq \taudiv(X_\umoins^j)/\bar \taudiv\}}\,
        \bigl[
                -\delta_{A_{t-u}(X^j_\umoins)}
                +\delta_{A_{t-u}(\alpha\,X^j_\umoins)}
\\[-1.2em]
\nonumber
&\qquad\qquad\qquad\qquad\qquad\qquad\qquad\qquad\qquad\qquad\qquad
                +\delta_{A_{t-u}((1-\alpha)\,X^j_\umoins)}
         \bigr]\,
		 		\MP_1(\dif u, \dif j, \dif \alpha, \dif \theta) 
\\[0.2em]
\label{da.def.proc.S.nu}
  &\qquad-
		\iint\limits_{[0,t]\times\NN^*} 
			1_{\{j\leq N_{\umoins}\}} \, 
			 \delta_{A_{t-u}(X^j_\umoins)}
			 			 \,	\MP_2(\dif u, \dif j)\, .
\end{align}
The first term on the right hand corresponds to the state of the population at time $t$ without division and death. At each division or death time, we modify the final state taking into account the discrete event.

%----------------------------------------
\begin{remark}
In this model, individuals do not interact with each other. For any $t>0$, the lineages generated by individuals $x^1_t, \dots, x^{N_t}_t$ are then independent and this process can be seen as a multitype branching process in continuous time and types.
\end{remark}
%----------------------------------------

We introduce the compensated Poisson random measures associated to $\MP_1$ and $\MP_2$ as
\begin{align*}
  \tilde \MP_1(\dif u, \dif j, \dif y, \dif \theta)
  &\eqdef
  \MP_1(\dif u, \dif j, \dif y, \dif \theta)
  -
  n_1(\dif u, \dif j, \dif y, \dif \theta)\,,
\\
 \tilde \MP_2(\dif u, \dif j)
  &\eqdef
  \MP_2(\dif u, \dif j)-n_2(\dif u, \dif j)\,.
\end{align*}

The process $(\eta_t)$ is a Markov process with the following infinitesimal generator (see \cite{campillo2014d})
\begin{align*}
  \L\Phi(\eta)
  & \eqdef 
  \crochet{\eta ,\rhog\, f'} \,
  F'(\crochet{\eta, f}) 
\\
  & \quad 
  +  
  \int_{0}^{\mmax} \taudiv(x) \,
     \int_0^1 			  
	   \Bigl[ 
	      F(\crochet{\textstyle \eta- \delta_{x}
	         +\delta_{\alpha\,x}
		    +
		    \delta_{(1-\alpha)\,x},f}
		  ) 
		 -
		  F(\crochet{\eta,f}) 
	   \Bigr]		\;	
			  q(\alpha)\,\dif \alpha  \, \eta(\dif x) 
\\
	&\quad	+ D\, \int_{0}^{\mmax} 
			\Bigl[
			  F(\crochet{\textstyle \eta- \delta_{x},f})
			  -
			  F(\crochet{\eta,f}) 
			\Bigr]\,	\eta(\dif x) 
\end{align*}
for any test function of the form
\[
   \Phi(\eta)=F(\crochet{\eta,f})
\]
with  $F \in C_b^{1}(\RR)$ et $f \in C^{1}[0,\mmax]$.

\medskip
Note that for $\eta=\sum_{i=1}^n \delta_{x_i}$,
$$
	\L \crochet{\eta,f} = \sum_{i=1} \op^*f(x_i)\,.
$$

From Ito's formula for semi-martingales with jumps (see for exemple \cite{protter2005a}), we obtain the following semimartingale decomposition (see \cite[Proposition 4.1.]{campillo2014d}).
 
%------------------------------------------
\begin{proposition}
\label{da.prop.semimartingale}
Let $F\in C^1(\RR)$ and $f: (t,x) \to f_t(x) \in C^{1,1}(\RR_+ \times [0,\mmax])$. For any $t>0$,
\begin{align*}
  &
  F(\crochet{\eta_t, f_t})
  = 
  F(\crochet{\eta_0, f_0})
   + \int_0^t  \left[
	 		 \crochet{ \eta_u , \, g\,\partial_x f_u+\partial_u f_u} \,
	 		F'\left(\crochet{\eta_u, f_u}\right)
	 	\right] \dif u 
\\  
  &\quad
  + \int_0^t \int_0^\mmax 
  		\taudiv(x)\, \int_0^1
		\bigl[
			F(\crochet{\eta_u
					-\delta_x
					+\delta_{\alpha \, x}
					+\delta_{(1-\alpha)\,x}, f_u} 
			)
        \\[-0.9em] 
        & \qquad\qquad\qquad\qquad\qquad\quad
                     \qquad\qquad\qquad\qquad\qquad
			-
			F(\crochet{\eta_u, f_u}) 
	    \bigr] q(\alpha)\, \dif \alpha \, \eta_u(\dif x)\,\dif u
\\
  &\quad  
  + D \, \int_0^t \int_0^\mmax  
	  \bigl[
		F(\crochet{\eta_u-\delta_x, f_u})
	    -
	    F\left( \crochet{\eta_u, f_u}\right) 
	  \bigr] \, \eta_u(\dif x)\,\dif u
\\[1em]
	& \quad
	  + M_t^{1,F,f}
	   + M_t^{2,F,f}
\end{align*}
with
\begin{align*}
M_t^{1,F,f}  
  &=
   \iiiint\limits_{[0,t]\times\NN^*\times[0,1]^2} 
		1_{\{j\leq N_{\umoins}\}} \,
		1_{\{\theta \leq \taudiv(X_\umoins^j)/
				\bar \taudiv\}}\, 
		\bigl[
			F(\crochet{\eta_{\umoins}
					-\delta_{X_\umoins^j}
					+\delta_{\alpha \, X_\umoins^j}
					+\delta_{(1-\alpha)\, X_\umoins^j}, f_u} 
			)
        \\[-0.9em] 
        \nonumber
        & \qquad\qquad\qquad\qquad\quad
                     \qquad\qquad\qquad\qquad\qquad
			-
			F(\crochet{\eta_{\umoins}, f_u}) 
	    \bigr] 
	    \;
		 \tilde \MP_1(\dif u, \dif j, \dif \alpha, \dif \theta)  
\\[0.7em]
M_t^{2,F,f}
  &= 
  \iint\limits_{[0,t]\times\NN^*} 
	  1_{\{j\leq N_{\umoins}\}} \, 
	  \bigl[
		F(\crochet{\eta_{\umoins}-\delta_{X_\umoins^j}, f_u})
	    -
	    F\left( \crochet{\eta_{\umoins}, f_u}\right) 
	  \bigr] \, \tilde \MP_2(\dif u, \dif j).
\end{align*}
\end{proposition}
%------------------------------------------

\medskip

From \cite[page 62]{ikeda1981a}, the following result holds.
%------------------------------------------
\begin{proposition}
\label{da.prop.martingales}
\begin{enumerate}
\item If for any $t\geq 0$:
\begin{align*}
&\EE \Bigl(
	\int_0^t \int_0^\mmax 
  		\taudiv(x)\, \int_0^1
		\bigl[
			F(\crochet{\eta_u
					-\delta_x
					+\delta_{\alpha \, x}
					+\delta_{(1-\alpha)\,x}, f_u} 
			)
\\
&\qquad\qquad\qquad\qquad\qquad\qquad\qquad
			-
			F(\crochet{\eta_u, f_u}) 
	    \bigr] q(\alpha)\,\dif \alpha\, \eta_u(\dif x)\,\dif u
	 \Bigr)< \infty
\end{align*}
then $(M_t^{1,F,f})_{t\geq 0}$ is a $(\F_t)_{t\geq0}$-martingale.

\item If for any $t\geq 0$:
\begin{align*}
\EE \Big(
	\int_0^t \int_0^\mmax  
	  \bigl[
		F(\crochet{\eta_u-\delta_x, f_u})
	    -
	    F\left( \crochet{\eta_u, f_u}\right) 
	  \bigr] \, \eta_u(\dif x)\,\dif u
	 \Big)< \infty
\end{align*}
then $(M_t^{2,F,f})_{t\geq 0}$ is a $(\F_t)_{t\geq0}$-martingale.
\end{enumerate}
\end{proposition}
%------------------------------------------

\medskip

We set $\PP_{\eta_0}$ the probability under the initial condition $\eta_0$.

\medskip

By the same argument used by \cite[Corollary 4.3.]{campillo2014d}, we can prove the following result.
%------------------------------------------
\begin{lemma}[Control of the population size]
\label{da.lem.taille.pop}
For any $t\geq 0$ and $\ell\geq 1$:
\begin{align*}
  \EE_{\delta_x}\left( 
    \sup_{u \in [0,t]} \crochet{ \eta_u,1 }^\ell 
  \right) 
  \leq C_{\ell,t}
\end{align*}
where $C_{\ell,t}<\infty $ only depends on $\ell$ and $t$.
\end{lemma}
%------------------------------------------

%%%%%%%%%%%%%%%%%%%%%%%%%%%%%%%%%%%%%%%%%%%%%%%%%%%%%%%%%%%%%%%%%%%%%%
%%%%%%%%%%%%%%%%%%%%%%%%%%%%%%%%%%%%%%%%%%%%%%%%%%%%%%%%%%%%%%%%%%%%%%
\section{Extinction probability of the growth-fragmentation-death individual based model}
\label{sec.extinction.proba}
%%%%%%%%%%%%%%%%%%%%%%%%%%%%%%%%%%%%%%%%%%%%%%%%%%%%%%%%%%%%%%%%%%%%%%
%%%%%%%%%%%%%%%%%%%%%%%%%%%%%%%%%%%%%%%%%%%%%%%%%%%%%%%%%%%%%%%%%%%%%%

We are interested in the survival probability of the population described in Section \ref{subsec.IBM}.

We suppose that, at time $t=0$, there is only one individual, with mass $x_0$, in the population, i.e.
$$
	\eta_0(\dif x) = \delta_{x_0}(\dif x)\,.
$$
 
Note that in age-structured models, all individuals have the same age 0 at birth. 
It allows to easily characterize the survival probability using standard results on Galton-Watson processes \citep{TranChi2009a}.
In our case, masses of individuals from a division can take any value in $(0,\mmax)$. The extinction probability of the population then depends on the mass of the initial individual and the survival condition is more complex to obtain.

\bigskip

The extinction probability of the population with initial mass $x_0$ is
\begin{align*}
  p(x_0) \eqdef \PP_{\delta_{x_0}}(\exists t>0, N_t=0)\,.
\end{align*}
The survival event will be denoted $\{\text{survival}\}=\{\exists t>0, N_t=0\}^c$ and the survival probability is $\PP_{\delta_{x_0}}(\text{survival})=1- p(x_0)$.

We define the $n$-th generation as the set of individuals descended from a division of one individual of the $(n-1)$-th generation.
The generation 0 corresponds to the initial population.

We introduce the following notation,
\begin{align*}
%\label{proba.extinction.n.gen}
  p_n(x_0) 
  \eqdef
  \PP_{\delta_{x_0}}
    (\text{extinction of the population before the $n$-th generation})\,,
    \quad n\in \NN\,.
\end{align*}
It is obvious that
$$
	\lim_{n \to \infty} p_n(x_0)=p(x_0)\,.
$$

\medskip

Let $\tau$ be the stopping time of the first event (division or death), then at time $\tau$ the population is given by
\begin{align}
\label{eta_tau}
  \eta_\tau
  \eqdef
  \left\{
			\begin{array}{ll}
				0\, & \text{if death}\,,\\
				\delta_{X_1} + \delta_{X_2}\, & \text{if division}\,,
			\end{array}
  \right.
\end{align}
with $X_1=\alpha \, A_\tau(x_0)$ and $X_2=(1-\alpha)\,A_\tau(x_0)$ where the proportion $\alpha$ is distributed according to the kernel $q(\alpha)$.

\bigskip

%----------------------------------------
\begin{proposition}
\label{da.prop.proba.extinction}
$p$ is the minimal non-negative solution of
\begin{align}
\nonumber
  p(x_0)
  &=
  \int_0^\infty D\, e^{-D\,t}\, e^{-\int_0^t \taudiv(A_u(x_0))\,\dif u}\, 
  \dif t
\\
\label{da.eq.proba.extinction}
  &\quad
  + \int_0^\infty \taudiv(A_t(x_0)) \, e^{-\int_0^t \taudiv(A_u(x_0)) 
      \,\dif u -D\,t}\,
	\int_0^1 q(\alpha)\,p\big(\alpha \, A_t(x_0)\big)\,
	     p\big((1-\alpha) \, A_t(x_0)\big)\, \dif \alpha\, \dif t \,,
\end{align}
in the sense that for any non-negative solution $\tilde p$ we have $\tilde p\geq p$.
\end{proposition}
%----------------------------------------

This result is similar to the expression of the extinction probability for a non-homogeneous linear birth and death process (see for example \citet{bailey1963a}).

%----------------------------------------
\begin{remark}
The function $f \equiv 1$ is solution to \eqref{da.eq.proba.extinction}. Hence, $0 \leq p(x) \leq 1$ for any $x\in[0,\mmax]$.
\end{remark}
%----------------------------------------

%----------------------------------------
\begin{lemma}
\label{lemme.proba.evnt}
\begin{enumerate}
\item The probability for one individual with mass $x_0$ to die before its division and before a given time $t_0$ is
$$
	\PP_{\delta_{x_0}}(\eta_\tau=0; \tau\leq t_0)
		=
	 \int_0^{t_0} D\, e^{-D\,t}\, e^{-\int_0^t \taudiv(A_u(x_0)) \dif u}\, \dif t \,.
$$

\item For any bounded measurable function $f:[0,\mmax]^2 \to \RR$:
\begin{align*}
	\EE_{\delta_{x_0}}\left[
						1_{\{ \eta_\tau \neq 0\}} \,1_{\{ \tau \leq t_0\}} \,
						f(X_1, X_2) 
					  \right]
&=
			\int_0^{t_0} 
				\taudiv(A_t(x_0))\, 
				e^{-D\,t}\, e^{-\int_0^t \taudiv(A_u(x_0)) \dif u}\,
\\
	& \qquad \qquad
				\int_0^1
					q(\alpha)\,
					f\big(\alpha \,A_t(x_0) , \, (1-\alpha) \,A_t(x_0)\Big)\, \dif \alpha				
				\, \dif t \, .
\end{align*}
\end{enumerate}
\end{lemma}
%----------------------------------------

%----------------------------------------
\begin{proof}
By construction of the process in \eqref{da.def.proc.S.nu}, we get
\[
  \tau = T_{1}\wedge T_{2}\,,
\]
where, $T_{1}$ is the time of the first jump of the process 
\[
  t\mapsto \MP_{1}(\Gamma_t(x_0))
\]
with 
$\Gamma_t(x_0)=
	\left\{
		(s,1,y,\theta),\, s\in[0,t],\,y\in [0,1], 
		\theta\in\left[0,\frac{\taudiv(A_s(x_0))}{\bar \taudiv}\right] 
	\right\}
$
and $T_{2}$ is the time of the first jump of the process 
\[
  t\mapsto \MP_{2}([0,t]\times \{1\})\,.
\]

%Si le premier événement est une division, alors $T_{d}$ correspond au premier instant de division sinon $T_{s}$ correspond au premier instant de soutirage.
The distribution of $T_{1}$ is a non-homogeneous exponential distribution with parameter $\taudiv(A_t(x_0))$, i.e. with the probability density function $\taudiv(A_t(x_0))\,\exp({-\int_0^t \taudiv(A_u(x_0)) \dif u})$. The distribution of $T_{2}$ is a (homogeneous) exponential distribution with parameter $D$. $T_{1}$ and $T_{2}$ are independent. We then deduce the results of the lemma.
\end{proof}
%----------------------------------------

%----------------------------------------
\begin{proof}[Proof of Proposition \ref{da.prop.proba.extinction}]
The extinction probability of the population is
\begin{align*}
  p(x_0)
  &= 
  \PP_{\delta_{x_0}}(\eta_{\tau}=0) 
  +
  \PP_{\delta_{x_0}}(\exists t>\tau, N_t=0\textrm{ and }\eta_{\tau}\neq 0) 
%\\ %----
%  &= 
%  \PP_{\delta_{x_0}}(\eta_{\tau}=0) 
%  +
%  \EE_{\delta_{x_0}}(
%     \indic_{\{\exists t>\tau, N_t=0\}}\,\indic_{\{\eta_{\tau}\neq 0\}}
%  )
\end{align*}

From the Markov property at time $\tau$, we get
\begin{align*}
&\PP_{\delta_{x_0}}(\exists t>\tau, N_t=0\textrm{ and }\eta_{\tau}\neq 0) 
=
  \EE_{\delta_{x_0}}\bigl(
     \EE_{\delta_{X_1}+\delta_{X_2}}(
      \indic_{\{\exists t>0, N_t=0\}})\,\indic_{\{\eta_{\tau}\neq 0\}}
  \bigr) 
\end{align*}
where the random variables $X_1$ et $X_2$ are defined by \eqref{eta_tau}. Since the two lineages are independent, then
\begin{align*}
&\PP_{\delta_{x_0}}(\exists t>\tau, N_t=0\textrm{ and }\eta_{\tau}\neq 0) 
  = 
  \EE_{\delta_{x_0}}\bigl(
     p(X_{1})\,p(X_{2})\,\indic_{\{\eta_{\tau}\neq 0\}}
  \bigr) \,.
\end{align*}
Hence \eqref{da.eq.proba.extinction} follows from Lemma \ref{lemme.proba.evnt}.

\medskip

We now prove that $p$ is the minimal solution of \eqref{da.eq.proba.extinction}.
Let $\tilde p$ be a non-negative solution of \eqref{da.eq.proba.extinction}. Then for any $x$, $0=p_0(x)\leq \tilde p(x)$. Following the previous approach, we can show that
\begin{multline*}
  p_n(x_0)
  =
  \int_0^\infty D\, e^{-D\,t}\, e^{-\int_0^t \taudiv(A_u(x_0))\,\dif u}\, 
  \dif t
  + \int_0^\infty \taudiv(A_t(x_0)) \, e^{-\int_0^t \taudiv(A_u(x_0)) 
      \,\dif u -D\,t}\,
\\
	\int_0^1 q(\alpha)\,p_{n-1}\big(\alpha \, A_t(x_0)\big)\,
	     p_{n-1}\big((1-\alpha) \, A_t(x_0)\big)\, \dif \alpha\, \dif t\,.
\end{multline*}
By recurrence, we then get $p_n(x_0) \leq \tilde p(x_0)$. Passing to the limit, we conclude that $p(x_0)=\lim_{n \to \infty}p_n(x_0)\leq \tilde p(x_0)$.
\end{proof}
%----------------------------------------

\bigskip

For any $x \in (0,M)$ and $y>0$ such that $x\leq y$, let $t(x,y)$ be the first hitting time of $y$ by the flow $A_t(x)$, i.e. 
\begin{align}
\label{temps.atteinte.masse}
  t(x,y)
  \eqdef
  \inf\{t \geq 0 ,\, A_t(x)=y \}
  =
  \left\{
			\begin{array}{ll}
				\tilde A_x^{-1}(y) \,, & \text{if $x\leq y<M$}\,,\\
				+\infty\,, & \text{if $y\geq M$}\,,
			\end{array}
  \right.
\end{align}
where $\tilde A_x^{-1}$ is the inverse function of the $C^1$-diffeomorphism $t\mapsto A_t(x)$.

\bigskip

%----------------------------------------
\begin{theorem}
\label{da.th.proba_inf1}
The following equivalence holds:
$$
	\exists x_0 \in (0,M), \, p(x_0)<1 \Longleftrightarrow \forall x \in (0,M), \, p(x)<1.
$$
\end{theorem}
%----------------------------------------

In other words, the survival probability of the population in an environment depends on the mass of the initial mass, however the fact that this environment is in favor of the population growth does not depends on its initial mass.

%----------------------------------------
\begin{proof}
Let $x_0 \in (0,M)$ be such that $p(x_0)<1$.
We consider three cases.
\begin{enumerate}
%---
\item \label{fenumi} $x\in(0,x_0)$. 

The time $t(x,x_0)$ is finite. There is survival in the population with initial mass $x$ if the initial individual reaches mass $x_0$ without division, and then survives.
We then get from the Markov property
\begin{align*}
  \PP_{\delta_{x}}(\text{survival}) 
  \geq 
  \exp\Bigl(
    -D\,t(x,x_0) -\int_0^{t(x,x_0)} \taudiv(A_u(x))\,\dif u
  \Bigr)\times 
			\PP_{\delta_{x_0}} (\text{survival})
   >0
\end{align*}
hence $p(x)<1$.

%---
\item \label{fenumii} $x\in[x_{0},x_{0}\vee \mdiv]$. 

As $x_0 \leq x \leq \mdiv$, the population with initial mass $x_0$ survives if and only if the initial individual with mass $x_0$ reaches the mass $x$ and the population survives. Then,
\begin{align*}
  \PP_{\delta_{x_0}}(\text{survival}) 
  =
  e^{-D\,t(x_0,x)}\, \PP_{\delta_{x}}(\text{survival})
\end{align*}
and
\begin{align*}
  p(x)
  =
  1- (1-p(x_0))\,e^{D\,t(x_0,x)} < 1\,.
\end{align*}

%---
\item $x\in(x_{0}\vee \mdiv,\mmax)$.

Let $x_1=x_0\vee\mdiv$. From the point \ref{fenumii}, $p(x_1)<1$.
For any individual with mass $z$ such that $x_1 \leq z < 2\,x_1 \wedge \mmax $, a survival possibility is to divide before it reaches the mass $2\,x_1 \wedge \mmax$, that the mass of the ``first'' child is less that $x_1$ and the lineage produced by this child survives. Then
\begin{align*}
\PP_{\delta_{z}}(\text{survival}) 
	& \geq 
		\int_0^{t(z,2\,x_1)} 
			\taudiv(A_t(z)) \, 
			e^{-\int_0^t \taudiv(A_u(z)) \dif u -D\,t}\,		
		\int_0^{x_1} 
				q\left(\frac{y}{A_t(z)} \right) \, 
				\frac{\PP_{\delta_{y}}(\text{survival})}{A_t(z)} \,
				\dif y \, \dif t\,.
\end{align*}
Moreover
\begin{align}
\label{proba.mass.offspring}
\int_0^{x_1} 
	\frac{1}{A_t(z)} \,
	q\left(\frac{y}{A_t(z)} \right) \, 
	\dif y \,
&= 
		\int_0^{\frac{x_1}{A_t(z)}} 
		q(\alpha) \, 
		\dif \alpha 
\geq
		\int_0^{\frac 12} 
		q(\alpha) \, 
		\dif \alpha 
= \frac 12
>0\,.
\end{align}
From the point \ref{fenumi},  $\PP_{\delta_{y}}(\text{survival})=1-p(y)>0$ for any $y<x_1$, then
$$
\int_0^{x_1} 
				q\left(\frac{y}{A_t(z)} \right) \, 
				\frac{\PP_{\delta_{y}}(\text{survival})}{A_t(z)} \,
				\dif y
	>0\,.
$$
Moreover, $\taudiv(A_t(z))>0$ for any $0 \leq t \leq t(x,2\,x_1)$ because $z>\mdiv$. Finally, we then have $p(z)<1$ for any $z\in(x_{0}\vee \mdiv, 2\,(x_{0}\vee \mdiv) \wedge \mmax)$. 

Iterating this argument, we can prove that $p(x)<1$ for any $x\in(x_{0}\vee \mdiv,\mmax)$.
\qedhere
\end{enumerate}
\end{proof}
%----------------------------------------

%%%%%%%%%%%%%%%%%%%%%%%%%%%%%%%%%%%%%%%%%%%%%%%%%%%%%%%%%%%%%%%%%%%%%%
%%%%%%%%%%%%%%%%%%%%%%%%%%%%%%%%%%%%%%%%%%%%%%%%%%%%%%%%%%%%%%%%%%%%%%
\section{Links between the stochastic and deterministic approaches}
\label{sec.lien.approches}
%%%%%%%%%%%%%%%%%%%%%%%%%%%%%%%%%%%%%%%%%%%%%%%%%%%%%%%%%%%%%%%%%%%%%%
%%%%%%%%%%%%%%%%%%%%%%%%%%%%%%%%%%%%%%%%%%%%%%%%%%%%%%%%%%%%%%%%%%%%%%

The aim of this section is to link the stochastic and deterministic descriptions of the possibility of survival of a population. Thinking in terms of adaptive dynamics, the main difference between these two descriptions is that the invasion fitness of the stochastic model is usually defined as the survival probability of the population whereas for the deterministic model, the invasion fitness is the exponential growth rate of the population.

\medskip

We assume that the following assumption holds (conditions ensuring existence of such a solution are given in Section \ref{sec.gfd.eigenpbl})).
\begin{hypothesis}
\label{existence.eigenelements}
The system \eqref{eq.eigenproblem}-\eqref{eq.eigenproblem.dual} admits a solution $(\hat u,\hat v, \Lambda)$ such that $\hat v \in C[0,M]\cap C^1(0,\mmax)$.
\end{hypothesis}
We consider the process
$$
	\eta_t=\sum_{i=1}^{N_t}\delta_{X_t^i}
$$ 
defined by \eqref{da.pop.modele.reduit}.

%%%%%%%%%%%%%%%%%%%%%%%%%%%%%%%%%%%%%%%%%%%%%%%%%%%%%%%%%%%%%%%%%%%%%%
\subsection{Stochastic process and eigenproblem}
\label{subsec.stoch.proc.eigenpbl}
%%%%%%%%%%%%%%%%%%%%%%%%%%%%%%%%%%%%%%%%%%%%%%%%%%%%%%%%%%%%%%%%%%%%%%

The following proposition gives a stochastic representation of the deterministic model.
%----------------------------------------
\begin{proposition}
\label{da.prop.generateur}
For $f\in C^1[0,\mmax]$, let $v\in C^{1,1}(\RR_+ \times (0,\mmax))$ be a solution of
\begin{align}
\label{da.eq.evolution}
		\frac{\partial}{\partial t} v(t,x)  =  \op^* v(t,x)\,,\ t\geq 0\,,\quad 
		v(0,x)  =  f(x)\,,x\in[0,\mmax]\,,
\end{align}
bounded on $[0,T]\times[0,\mmax]$ for any $T>0$, where $\op^*$ is defined by \eqref{def.G.star}. Then 
$$
	v(t,x) = \EE_{\delta_{x}}\left[ \sum_{i=1}^{N_t}f\left(X_t^i \right)\right]\,.
$$

\end{proposition}
%----------------------------------------

%----------------------------------------
\begin{proof}
Let $t\geq 0$ be fixed.  For any $s \leq t$, we define
\begin{align*}
  M_s
  & =
  \sum_{i=1}^{N_s} v(t-s, X_s^i).
\end{align*}
Proposition \ref{da.prop.semimartingale} gives
\begin{align*}
  &M_s
  =
  v(t,X_0)
  +
  \int_0^s \left(
    \sum_{i=1}^{N_u}g(X_u^i)\,\partial_x v(t-u,X_u^i) 
    - 
    \partial_t v(t-u,X_u^i)
  \right)\, \dif u
\\
  &\quad 
  + \iint\limits_{[0,t]\times[0,1]} 
    \sum_{j=1}^{N_u} \taudiv(X_\umoins^j) 
	   \bigl[
			v(t-u,\alpha X_{\umoins}^j)
			+v(t-u,(1-\alpha)\, X_{\umoins}^j)
        \\[-1.2em] 
        \nonumber
        & \qquad\qquad\qquad\qquad\qquad\quad
                     \qquad\qquad\qquad\qquad\qquad
			-
			v(t-u,X_{\umoins}^j)  
	    \bigr] q(\alpha)\,
		 \dif \alpha \dif u  
\\
  &\quad  
  -D\, \int_0^s 
	 \sum_{j=1}^{N_u} 
		v(t-u,X_u^j)
		 \, \dif u
\\
		&\quad + \iiiint\limits_{[0,s]\times\NN^*\times[0,1]^2} 
		1_{\{j\leq N_{\umoins}\}} \,
		1_{\{\theta \leq \taudiv(X_\umoins^j)/
				\bar \taudiv\}}\, 
		\bigl[
			v(t-u,\alpha X_{\umoins}^j)
			+v(t-u,(1-\alpha)\, X_{\umoins}^j)
        \\[-1.2em] 
        \nonumber
        & \qquad\qquad\qquad\qquad\qquad\quad
                     \qquad\qquad\qquad\qquad\qquad
			-v(t-u,X_{\umoins}^j) 
	    \bigr] 
	    \;
		 \tilde \MP_1(\dif u, \dif j, \dif \alpha, \dif \theta)  
\\
  &\quad  
  - \iint\limits_{[0,s]\times\NN^*} 
	  1_{\{j\leq N_{\umoins}\}} \, 
      v(t-u,X_{\umoins}^j)	  
	   \, \tilde \MP_2(\dif u, \dif j)\,.
\end{align*}
Then
\begin{align*}
  M_s
  & =
		 v(t,X_0)
		 +
		 \int_0^s \sum_{j=1}^{N_u}\left(
		 			\op^* v(t-u,X_u^i) 
		 			- \partial_t v(t-u,X_u^i)
		 			\right)\, \dif u
\\
		&\qquad + \iiiint\limits_{[0,s]\times\NN^*\times[0,1]^2} 
		1_{\{j\leq N_{\umoins}\}} \,
		1_{\{\theta \leq \taudiv(X_\umoins^j)/
				\bar \taudiv\}}\, 
		\bigl[
			v(t-u,\alpha X_{\umoins}^j)
			+v(t-u,(1-\alpha)\, X_{\umoins}^j)
        \\[-1.2em] 
        \nonumber
        & \qquad\qquad\qquad\qquad\qquad\quad
                     \qquad\qquad\qquad\qquad\qquad
			-v(t-u,X_{\umoins}^j) 
	    \bigr] 
	    \;
		 \tilde \MP_1(\dif u, \dif j, \dif \alpha, \dif \theta)  
\\
  &\qquad  
  - \iint\limits_{[0,s]\times\NN^*} 
	  1_{\{j\leq N_{\umoins}\}} \, 
      v(t-u,X_{\umoins}^j)	  
	   \, \tilde \MP_2(\dif u, \dif j).
\end{align*}

As $v$ is a solution of \eqref{da.eq.evolution}, the second term equals zero. By assumption, $v$ is bounded on $[0,t]\times [0,\mmax]$. From Proposition \ref{da.prop.martingales} and Lemma \ref{da.lem.taille.pop}, $(M_s)_{s\geq 0}$ is then a martingale. Thus $\EE_{\delta_{x}}(M_t)=\EE_{\delta_{x}}(M_0)$ and the conclusion follows.
\end{proof}
%----------------------------------------

\medskip

The function $v(t,x)=e^{\Lambda\,t} \, \hat v(x)$ is a trivial solution of \eqref{da.eq.evolution} with initial condition $v(0,x)=\hat v(x)$. 
We then deduce the following result.

%----------------------------------------
\begin{corollary}
\label{da.prop.semi.groupe}
For any $t>0$,
\begin{align*}
	\EE_{\delta_{x}}\left[\sum_{i=1}^{N_t} \hat v(X_t^i) \right] = e^{\Lambda\,t} \, \hat v(x).
\end{align*}
\end{corollary}
%----------------------------------------

\medskip

The previous proposition allows to obtain convergence results of solutions of \eqref{da.eq.evolution}. 

%----------------------------------------
\begin{corollary}
\label{da.vitesse.cv}
Let $f\in C^1[0,M]$ such as there exist finite constants $C^->0$ and $C^+>0$ such as, for any $x\in(0,M)$,
\[ 
 C^-\,\hat v(x)\leq f(x) \leq C^+\,\hat v(x)\,.
\]
Let $v \in C^{1,1}(\RR_+ \times [0,M])$ be a solution of \eqref{da.eq.evolution} with initial condition $f$. Then, for any $x\in (0,M)$:
\begin{align*}
  C^- \hat v(x) 
  \leq 
  v(t,x)\,e^{-\Lambda\,t}  \leq C^+\,\hat v(x)\,.
\end{align*}
\end{corollary}
%----------------------------------------

%----------------------------------------
\begin{proof}
From Proposition \ref{da.prop.generateur} and Corollary \ref{da.prop.semi.groupe}, we have
\begin{align*}
v(t,x)
	&=
		\EE_{\delta_{x}}\left[ \sum_{i=1}^{N_t}f\left(X_t^i \right)\right]
	\leq
		C^+\,\EE_{\delta_{x}}\left[ \sum_{i=1}^{N_t}\hat v \left(X_t^i \right)\right]
	\leq
		C^+\,e^{\Lambda\,t}\,\hat v(x) \,.
\end{align*}
The right inequality is proved similarly.
\end{proof}
%----------------------------------------

%----------------------------------------
\begin{remark}
The previous corollary can also be proved using deterministic methods. For example, for $\Lambda>0$, we can use the generalized relative entropy following the approach described in \citet[Section 4.2.]{perthame2007a}.
\end{remark}
%----------------------------------------

%----------------------------------------
\begin{proposition}
\label{da.prop.hat_v.positive}
The function $\hat v$ is positive on $(0,\mmax)$.
%$$
%	\hat v(x)>0, \qquad x\in(0,\mmax).
%$$
\end{proposition}
%----------------------------------------

The proof follows a similar approach as for Theorem \ref{da.th.proba_inf1}.

%----------------------------------------
\begin{proof} 
Let $x_0 \in (0,\mmax)$ be such that $\hat v(x_0)>0$ (this $x_0$ exists because $\int_0^\mmax \hat v(x)\, \hat u(x) \dif x = 1$).

For any $x \in (0,\mmax)$ and $y\geq x$, $t(x,y)$ is defined by \eqref{temps.atteinte.masse}.
\begin{enumerate}
%---
\item Let $0<x<x_0$. The probability $p_0$ that one individual with mass $x$ reaches the mass $x_0$ without division and death is  
\begin{align*}
	p_0 
	= 
	\exp\Bigl(
	  {-D\,t(x,x_0) -\int_0^{t(x,x_0)} \taudiv(A_u(x))\,\dif u}
	\Bigr) 
	>0\,.
\end{align*}
From Corollary \ref{da.prop.semi.groupe} we then get
$$
 e^{\Lambda\,t(x,x_0)}\,\hat v(x)
 	= \EE_{\delta_{x}}\left[\sum_{i=1}^{N_{t(x,x_0)}} \hat v(X_{t(x,x_0)}^i) \right]
 	\geq p_0 \, \hat v(x_0)
 	>0.
$$
Hence $\hat v$ is positive on $(0, x_0]$.

%---
\item Let $x\in[x_{0},x_{0}\vee \mdiv]$. 
The population with initial mass $x_0$ survives until the time $t(x_0, x)$ if and only if the individual with mass $x_0$ reaches the mass $x$. We then get
\begin{align*}
0< e^{\Lambda\,t(x_0,{ x})}\,\hat v(x_0)
  &=  
  \EE_{x_0}\left[\sum_{i=1}^{N_{t(x_0,{ x})}} \hat v(X_{t(x_0,{ x})}^i) \right]
	=
  e^{-D\,t(x_0,{ x})}\, \hat v(x)\, .
\end{align*}

%---
\item Let $x\in(x_{0}\vee \mdiv,\mmax)$.
Let $x_0\vee\mdiv \leq x_1  < \mmax $ be such that $\hat v(x_1)>0$.
In the population with initial mass $z$, we have $N_{t(z,2\,x_1)}>0$ if the initial individual divides before $t(z,2\,x_1)$ and the lineage produced by the smallest individual survives.
We then get the following lower bound:
\begin{align*} 
e^{\Lambda\,t(z,2\,x_1)}\,\hat v(z)
=&
  \EE_{z}\left[\sum_{i=1}^{N_{t(z,2\,x_1)}} \hat v(X_{t(z,2\,x_1)}^i) \right]
\\
  \geq & 
  \int_0^{t(z,2\,x_1)} 
			\taudiv(A_t(z)) \, 
			e^{-\int_0^t \taudiv(A_u(z)) \dif u -D\,t}\,
\\
	& \qquad		
		\int_0^{x_1} 
				\frac{1}{A_t(z)}
				q\left(\frac{y}{A_t(z)} \right) \, 
					\EE_{y}
						\left[
							\sum_{i=1}^{N_{t(z,2\,x_1)-t}} \hat v(X_{t(z,2\,x_1)-t}^i) 
						\right] \,
				\dif y \, \dif t\,.
\end{align*}
From Corollary \ref{da.prop.semi.groupe} and Point \ref{fenumi}, for any $y \leq x_1$:
$$
	\EE_{y}\left[\sum_{i=1}^{N_{t(z,2\,x_1)-t}} \hat v(X_{t(z,2\,x_1)-t}^i)\right]
	=
	e^{\Lambda\,(t(z,2\,x_1)-t)}\, \hat v(y)>0 \,.
$$
By \eqref{proba.mass.offspring} and since $\taudiv(A_t(z))>0$ for any $0 \leq t \leq t(x,2\,x_1)$ because $z>\mdiv$, we have $\hat v(z)>0$.

\medskip

The result follows by applying recursively the last argument.
\qedhere
\end{enumerate}
\end{proof}
%----------------------------------------

%----------------------------------------
\begin{remark}
The previous proposition can also be proved using a deterministic approach as the one of \cite[Lemma 1.3.1]{Doumic2010a}.
\end{remark}
%----------------------------------------

\medskip

%----------------------------------------
\begin{corollary}
\label{da.corollaire.unicite.vp}
Let $(\hat u_1,\hat v_1,\Lambda_1)$ and $(\hat u_2, \hat v_2,\Lambda_2)$ be two solutions of \eqref{eq.eigenproblem}-\eqref{eq.eigenproblem.dual}, then $\Lambda_1=\Lambda_2$.
\end{corollary}
%----------------------------------------

%----------------------------------------
\begin{proof}
Given that $\op^* \hat v_2=\Lambda_2\,\hat v_2$ and $\op \hat u_1=\Lambda_1\,\hat u_1$, we have
\begin{align*}
\Lambda_2\,\int_0^M \hat u_1(x)\,\hat v_2(x)\,\dif x
	&=
		\int_0^M \hat u_1(x)\,\op^* \hat v_2(x) \, \dif x
\\
	&=
		\int_0^M \op \hat u_1(x)\,\hat v_2(x) \, \dif x
\\
	&= 
		\Lambda_1 \, \int_0^M \hat u_1(x)\,\hat v_2(x)\,\dif x\,.
\end{align*}
From Proposition \ref{da.prop.hat_v.positive}, $\hat v_2(x)>0$ for any $x\in(0,M)$, thus
$
 \int_0^M \hat u_1(x)\,\hat v_2(x)\,\dif x>0
$
and the conclusion follows.
\end{proof}
%----------------------------------------

\bigskip

We introduce notations which will be useful is some proofs.
For any $0<\varepsilon<\frac{\mmax}{2}$, we set
\begin{itemize}
\item[\textbullet] $L_t^\varepsilon$: the number of individuals with mass in $[0,\varepsilon)$ at time $t$,
\item[\textbullet] $M_t^\varepsilon$: the number of individuals with mass in $[\varepsilon,\mmax-\varepsilon]$ at time $t$,
\item[\textbullet] $K_t^\varepsilon$: the number of individuals with mass in $(\mmax-\varepsilon,\mmax]$ at time $t$,
\end{itemize}
so that $N_t = L_t^\varepsilon+M_t^\varepsilon+K_t^\varepsilon$.

%----------------------------------------
\begin{remark}
\label{remark.hat.v}
As we will see in Section \ref{sec.proof.th.pb.adjoint}, under convenient assumptions, the solution of
$$
 	\phi(x) 
 		=
 		2\, \int_0^\infty e^{-\int_0^t \Lambda+D+b(A_s(x)) \, \dif s}\, b(A_t(x))\,
 			\int_0^1q(\alpha)\,\phi(\alpha\,A_t(x))\,\dif \alpha\,\dif t
$$
is solution of \eqref{eigenequation.dual}. 
So $\phi(0)=0$ and if $b(\mmax)=0$ then $\phi(\mmax)=0$. The eigenfunction $\hat v$ is positive on $(0,\mmax)$ but might vanish at $0$ and $\mmax$.
That is the reason why we separate $[0,\mmax]$ in three intervals : $[0,\varepsilon)$, $[\varepsilon,\mmax-\varepsilon]$ and $(\mmax-\varepsilon,\mmax]$.
\end{remark}
%----------------------------------------

%----------------------------------------
\begin{proposition}
\label{majoration.esp.Nt}
For any $x\in (0,M)$, there exists $C_x>0$ such as
\[
  \EE_{\delta_{x}}(N_t) \leq C_x\, e^{\Lambda\,t}
  \,,\quad \forall t\geq 0\,.
\]
\end{proposition}
%----------------------------------------

%----------------------------------------
\begin{proof}
Let $\varepsilon>0$ be such as $\varepsilon < x <\mmax-\varepsilon$. 
Since $\hat v$ is continuous, it follows from Proposition \ref{da.prop.hat_v.positive}, that there exists $C_\varepsilon>0$ such that
\[
	1 
	\leq 
	C_\varepsilon \, \hat v(y)\, , \quad \forall y\in[\varepsilon, \mmax-\varepsilon ] \,.
\]

From Corollary \ref{da.prop.semi.groupe},
\begin{align}
\label{ineq.Mt.varepsilon}
	\EE_{\delta_{x}}(M_t^\varepsilon) 
		=
			\EE_{\delta_{x}} \left[ 
					\sum_{i=1}^{N_t} 1_{\{[\varepsilon, M-\varepsilon]\}}(X_t^i)
				   \right]
		\leq 
			C_\varepsilon \, 
			\EE_{\delta_{x}} \left[ 
					\sum_{i=1}^{N_t} \hat v \left(X_t^i \right)
				   \right]
		= 
			C_\varepsilon \, e^{\Lambda \, t} \, \hat v(x) \, .
\end{align}

Let $x_0 \in (\mmax/2 \wedge x,\, \mmax-\varepsilon)$.
Backward in time, as long as one of the ancestors of two different individuals with mass larger than $\mmax-\varepsilon$ at time $t$ have not reached mass $\mmax/2$, their lineages cannot coalesce. Hence, for each of the $K_t^\varepsilon$ individuals with mass larger than $\mmax-\varepsilon$ at time $t$, following its lineage backward in time, there is necessarily a time interval of length larger than $t(x_0,\mmax-\varepsilon)$ during with the ancestor has mass in $[x_0,\mmax-\varepsilon]$ and satisfies the following property $(P)_{s,t}$ : an individual $i$ at time $s\leq t$ satisfies property $(P)_{s,t}$ if it survives before dividing until time $t$ or it divides at some time $s \leq s' \leq t$ and its child with larger mass satisfies the property $(P)_{s',t}$.
In other words, there is no death before time $t$ in the lineage formed where we follow only the offspring with larger mass at each division. 
See Figure \ref{fig.larger.lineage}.

\begin{figure}
\begin{center}
\includegraphics[width=11cm]{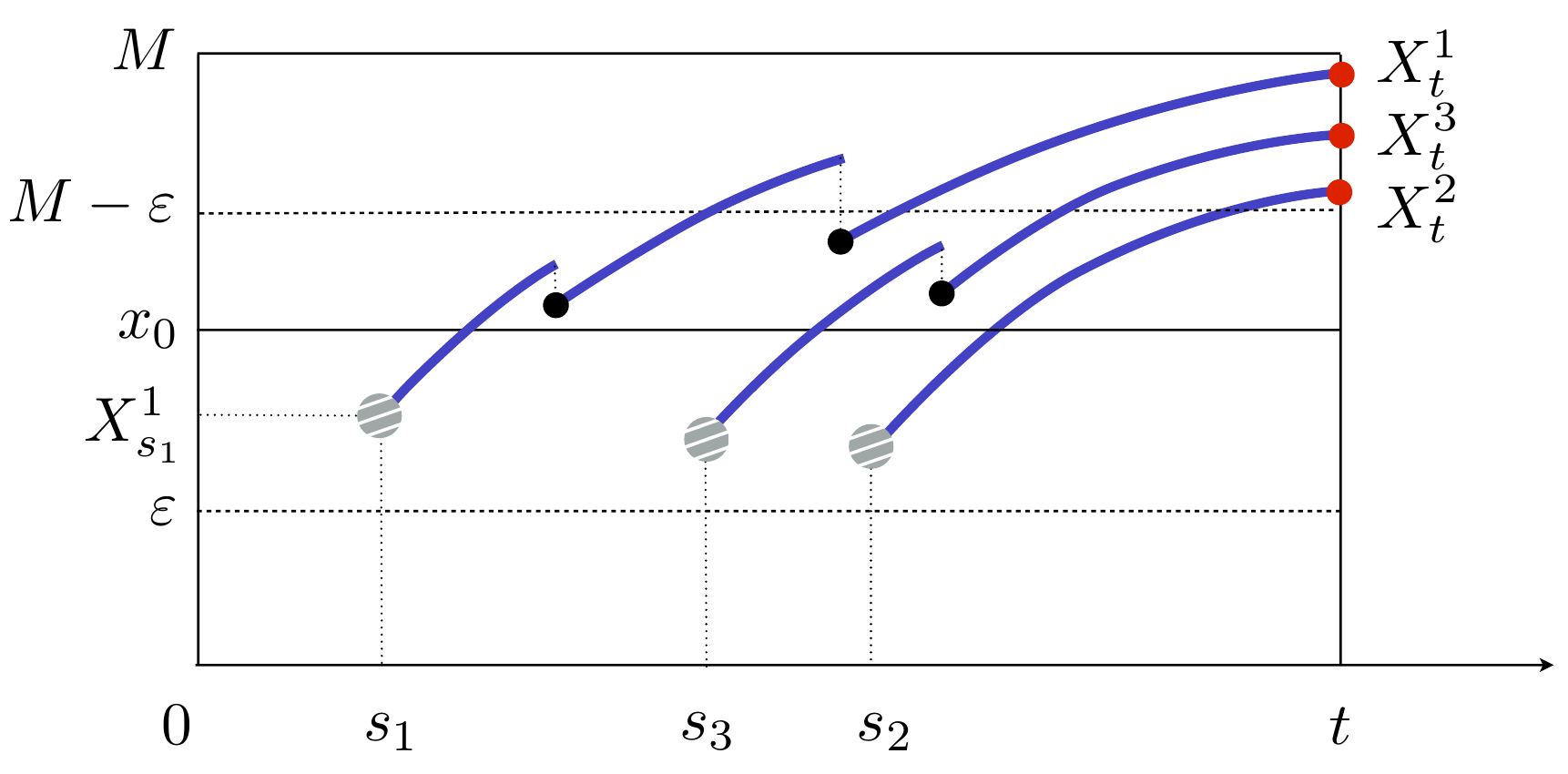}
\end{center}
\caption{\label{fig.larger.lineage}Representation of the backward lineages for individuals in $(\mmax-\varepsilon,\mmax]$ at time $t$, until they reach a mass smaller than $x_0$. Note that the three hatched individuals satisfy property $(P)_{s_i,t}$ for $i=1,2,3$.}
\end{figure}

Then
\begin{align*}
K_t^\varepsilon
	\leq &
	\frac{1}{t(x_0,\mmax-\varepsilon)}\,
	 \sum_{i=1}^{K_t^\varepsilon}
	 	\int_0^t
			1_{\{
				(\hat X_s^i \in (\varepsilon, M-\varepsilon))
				\text{ and $\hat X_s^i$ satisfies $(P)_{s,t}$}	
			 \}}
		\,
		\dif s\,,
\end{align*}
where $\hat X_s^i$ is the mass of the ancestor at time $s$ of the $i$-th individual with mass larger than $\mmax-\varepsilon$ at time $t$. In particular,
\begin{align*}
K_t^\varepsilon
	\leq &
	\frac{1}{t(x_0,\mmax-\varepsilon)}\,
	 \int_0^t
	 	\sum_{i=1}^{N_s}
			1_{\{
				(X_s^i \in (\varepsilon, M-\varepsilon))
				\text{ and $X_s^i$ satisfies $(P)_{s,t}$}	
			 \}}
		\,
		\dif s\,.
\end{align*}
We deduce from the Markov property that
\begin{align*}
\EE_{\delta_{x}}(K_t^\varepsilon) 
	&\leq 
	 	\frac{1}{t(x_0,\mmax-\varepsilon)}\,
		\int_0^t \EE_{\delta_{x}}
			\Big(\sum_{i=1}^{N_s}
			1_{\{
				(X_s^i \in (\varepsilon, M-\varepsilon))	
			 \}}\,
			 \PP_{X_s^i}\left(\text{the ancestor satisfies } (P)_{0,t-s} \right)
			 \Big)
			  \, \dif s
\\
	&=
	 	\frac{1}{t(x_0,\mmax-\varepsilon)}\,
		\int_0^t
		 	e^{-D(t-s)} \,
			\EE_{\delta_{x}}
			\Big(\sum_{i=1}^{N_s}
			1_{\{
				(X_s^i \in (\varepsilon, M-\varepsilon))	
			 \}}
			 \Big)
			  \, \dif s
\\
	& \leq	
		\frac{1}{t(x_0,\mmax-\varepsilon)}\,C_\varepsilon \, \hat v(x)
		\int_0^t
			e^{\Lambda \, s} \, e^{-D(t-s)}\, \dif s
\\
	&\leq	
		\frac{C_\varepsilon}{t(x_0,\mmax-\varepsilon)\,(\Lambda+D)}\, \, \hat v(x)
		e^{\Lambda\,t} 
\end{align*}
where the second inequality follows from \eqref{ineq.Mt.varepsilon} and the last one from the fact that $\Lambda+D>0$.

From Assumptions \ref{da.hypo.model.reduit}, $\taudiv$ is a continuous function such that $\taudiv(0)=0$. Furthermore $\Lambda+D>0$. Reducing $\varepsilon$ if necessary, we can then assume that $\taudiv_0 \eqdef \max_{0\leq x\leq \varepsilon} \taudiv(x) < \Lambda+D$. 
Then $L_t^\varepsilon$ is dominated by a branching process with birth rate $\taudiv_0$, death rate $D$ and inhomogeneous immigration rate $2\,\bar \taudiv\,(K_t^\varepsilon+M_t^\varepsilon)$.

Let $\bar L_t^\varepsilon$ be a such process with initial condition $\bar L_0^\varepsilon=0$, then
\begin{align*}
\frac{\dif}{\dif t} \EE_{\delta_{x}}(\bar L_t^\varepsilon)
	& =
		(\taudiv_0-D)\, \EE_{\delta_{x}}(\bar L_t^\varepsilon)
		+
		2 \,\bar \taudiv \,\EE_{\delta_{x}}(K_t^\varepsilon+M_t^\varepsilon)\,.
\end{align*}
Hence 
\begin{align*}
\EE_{\delta_{x}}(\bar L_t^\varepsilon)
	& =	
		e^{(\taudiv_0-D)\,t} \,		
		2 \,\bar \taudiv \int_0^t \EE_{\delta_{x}}(K_s^\varepsilon+M_s^\varepsilon)\,
		e^{(D-\taudiv_0)\,s}\, \dif s
\\
	& \leq
		 e^{(\taudiv_0-D)\,t} \,		
		2 \,\bar \taudiv \, C\,\hat v(x) \, \int_0^t e^{(\Lambda +D-\taudiv_0)\,s}\, \dif s
\\
	& \leq
		e^{(\taudiv_0-D)\,t} \,		
		2 \,\bar \taudiv \, C\, \hat v(x) \,\frac{e^{(\Lambda +D-\taudiv_0)\,t}-1}{\Lambda +D-\taudiv_0} \,.
\end{align*}

Finally we then get
\begin{align*}
\EE_{\delta_{x}}(N_t)
	&=
		\EE_{\delta_{x}}(K_t^\varepsilon+M_t^\varepsilon+L_t^\varepsilon)
	\leq C \, e^{\Lambda \, t}\,\hat v(x) \,.
\qedhere
\end{align*}
\end{proof}
%----------------------------------------

%%%%%%%%%%%%%%%%%%%%%%%%%%%%%%%%%%%%%%%%%%%%%%%%%%%%%%%%%%%%%%%%%%%%%%
\subsection{Invasion fitness}
\label{subsec.invasion.fitness}
%%%%%%%%%%%%%%%%%%%%%%%%%%%%%%%%%%%%%%%%%%%%%%%%%%%%%%%%%%%%%%%%%%%%%%

This section is dedicated to the proof of the following theorem, which links the two critera of survival possibility.

%----------------------------------------
\begin{theorem}
\label{th.critere}
\begin{enumerate}
\item \label{surcritique} If $\Lambda>0$ then $\PP_{\delta_{x}}(\text{survival})>0$ for any $x\in(0,\mmax)$. 
\item \label{sous-critique} If $\Lambda\leq 0$ then $\PP_{\delta_{x}}(\text{survival})=0$ for any $x\in[0,\mmax]$.
\item \label{evt.survie} Under the probability $\PP_{\delta_x}$, the following equality holds:
$$
	\{\text{survival} \}
	=
	\{\mathcal{Z}>0\} \quad \text{a.s.}
$$
where $\mathcal{Z}$ is an integrable random variable such that
\begin{align}
\label{limit.Z}
 \crochet{ e^{-\Lambda \, t}\, \eta_t, \, \hat v} 
 \xrightarrow[t\to\infty]{} \mathcal{Z}\ \textrm{ $\PP_{\delta_{x}}$-a.s.}
\end{align}
\end{enumerate}
\end{theorem}
%----------------------------------------

%----------------------------------------
\begin{remark}
The case $\Lambda>0$ (resp. $\Lambda=0$, $\Lambda<0$) is the super-critical case (resp. critical, sub-critique case) of the process $(\eta_t)_{t\geq 0}$ (see for example \citet{englander2004}).
\end{remark}
%----------------------------------------

%----------------------------------------
\begin{lemma}
\label{lemme.martingale.bornee}
For any $x \in (0,\mmax)$, $(\crochet{e^{-\Lambda \, t}\, \eta_t, \, \hat v})_{t\geq 0}$ is a $\PP_{\delta_{x}}$-martingale. Moreover if $\Lambda>0$ then it is bounded in $L^2(\Omega)$.
\end{lemma}
%----------------------------------------

%----------------------------------------
\begin{proof}
The function $\hat v$ is in $C[0,\mmax]$, thus it is bounded by a constant $C_{\hat v}>0$. Therefore, for any $t \geq 0$,
$$
	\EE_{\delta_{x}}(\crochet{e^{-\Lambda \, t}\, \eta_t, \, \hat v})
	\leq C_{\hat v} \, e^{-\Lambda \, t}\, \EE_{\delta_{x}}(N_t)\,,
$$
which is finite according to Lemma \ref{da.lem.taille.pop}, thus $\crochet{e^{-\Lambda \, t}\, \eta_t, \, \hat v}$ is integrable.

Let $s>0$,
\begin{align*}
\EE_{\delta_{x}}\Bigl(
		\crochet{ e^{-\Lambda \, (t+s)}\, \eta_{t+s}, \, \hat v}  
		\, \Big| \, \eta_t \Bigr)
	&= 
		\EE_{\delta_{x}}\Bigl( 
				e^{-\Lambda\,(t+s)} \, 
				\sum_{i=1}^{N_{t+s}} \hat v(X_{t+s}^i) \, \Big| \, \eta_t
				\Bigl)
\\
	&= 
		e^{-\Lambda\,(t+s)} \, \sum_{j=1}^{N_t} \,
		\EE_{\delta_{x}}\Bigl(  
			\sum_{i=1}^{N_{t+s}^j} \hat v(X_{t+s}^{i,j}) \, \Big| \, \eta_t
		\Bigr)
\end{align*}
where $\{X_{t+s}^{i,j}, i=1,\dots,N_{t+s}^j\}$ is the set of descendants, at time $t+s$, of the individual $X_t^j$.

From the Markov property, the independence of the lineages of $X_t^1,\dots,X_t^{N_t}$ and from  Corollary \ref{da.prop.semi.groupe},
\begin{align*}
  \EE_{\delta_{x}}\Bigl(
		\crochet{e^{-\Lambda \, (t+s)}\, \eta_{t+s}, \, \hat v} 
		\, \Big| \, \eta_t \Bigr)
	&= 
		e^{-\Lambda\,(t+s)} \, \sum_{j=1}^{N_t} \,
				\EE_{\delta_{X_t^j}}\Bigl(  
						\sum_{i=1}^{N_{s}^j} \hat v(X_{s}^{i,j})
						\Bigr)
\\
	&= 
		e^{-\Lambda\,(t+s)} \, \sum_{j=1}^{N_t} e^{\Lambda\,s}\,\hat v(X_t^j)			
\\
	&= 	
		\crochet{e^{-\Lambda \, t}\, \eta_t, \, \hat v} \,.
\end{align*}
Then $\crochet{e^{-\Lambda \, t}\, \eta_t, \, \hat v}$ is a martingale.
\medskip

Let $\delta>0$, we set for $n\geq 0$
\begin{align*}
  \sigma_n(x)
  &\eqdef
  \var_{\delta_{x}} \left( \sum_{i=1}^{N_{n \, \delta}} \hat v(X_{n\,\delta}^i)\right)
  \,.
\end{align*}
From the variance decomposition formula,
\begin{align*}
  \sigma_{n+1}(x)
  &=
  \EE_{\delta_{x}}\Biggl\{
     \var_{\delta_{x}}\biggl(
        \sum_{i=1}^{N_{(n+1)\,\delta}} \hat v(X_{(n+1)\,\delta}^i)
        \, \bigg|
        \, \eta_{n\, \delta}
     \biggr)
  \Biggr\}
\\
  &\qquad\qquad\qquad\qquad
  +
  \var_{\delta_{x}} \Biggl\{
    \EE_{\delta_{x}} \biggl(
      \sum_{i=1}^{N_{(n+1)\,\delta}} \hat v(X_{(n+1)\,\delta}^i) 
      \, \bigg|
      \, \eta_{n\,\delta}
    \Biggr)
  \Biggr\}\,.
\end{align*}
Given that $\crochet{ e^{-\Lambda \, t}\, \eta_t, \, \hat v}$ is a martingale, 
\begin{align*}
  \var_{\delta_{x}} \Biggl\{
    \EE_{\delta_{x}} \biggl(
			\sum_{i=1}^{N_{(n+1)\,\delta}}
			\hat v(X_{(n+1)\,\delta}^i) 
			\, \bigg| 
			\, \eta_{n\,\delta}
    \biggr)
  \Biggr\}
  &=
  \var_{\delta_{x}} \biggl\{
    e^{\Lambda\,\delta} \sum_{i=1}^{N_{n\,\delta}} 
    \hat v(X_{n\,\delta}^i)
  \biggr\}
=
  e^{2 \Lambda\,\delta} \, \sigma_n(x)\,.
\end{align*}
From the independence of the lineages of $X_t^1,\dots,X_t^{N_t}$ and the Markov property,
\begin{align*}
  \var_{\delta_{x}} \biggl(
    \sum_{i=1}^{N_{t+s}} \hat v(X_{t+s}^i) 
    \, \bigg| 
    \, \eta_t
  \biggr)
  &=
  \var_{\delta_{x}} \biggl(
    \sum_{j=1}^{N_{t}}\sum_{i=1}^{N_{t+s}^j} 
			\hat v(X_{t+s}^{i,j}) 
	\, \bigg| 
	\, \eta_t
	\biggr)
\\
  &=
  \sum_{j=1}^{N_t} \var_{\delta_{X_t^j}} 
    \biggl(\sum_{i=1}^{N_s} \hat v(X_s^i) \biggr)\,.
\end{align*}
Using the upper bound on $\hat v$,
\begin{align*}
  \var_{\delta_{x}} \biggl(
    \sum_{i=1}^{N_{t+s}} \hat v(X_{t+s}^i) \, \bigg| \, \eta_t
  \biggr)
  \leq
  C_{\hat v}\, \sum_{i=1}^{N_t} 
			\EE_{\delta_{X_t^i}}	(N_s^2)\,.
\end{align*}
From Lemma \ref{da.lem.taille.pop}, there exists $C_s>0$ such as:
\begin{align*}
  \EE_{\delta_{X_t^i}}	( N_s^2)  < C_s\,.
\end{align*}
Then
\begin{align*}
  \var_{\delta_{x}} \biggl(
     \sum_{i=1}^{N_{t+s}} \hat v(X_{t+s}^i) \, \bigg| \, \eta_t\biggr)
  \leq
  C_{\hat v}\, C_s \, N_t\,.
\end{align*}
We deduce by induction that
\begin{align*}
\sigma_{n+1}(x)
	&\leq
		C_{\hat v}\, C_\delta \,  \EE_{\delta_{x}}(N_{n\,\delta})
		+ e^{2 \Lambda\,\delta} \, \sigma_n(x)\,.
\end{align*}
It follows from Proposition \ref{majoration.esp.Nt} that $\EE_{\delta_{x}}(N_{n\,\delta}) \leq C_x\, e^{\Lambda \,n\,\delta}$ and thus
\begin{align*}
\sigma_{n+1}(x)
	\leq	
		C_{\delta,x}\,e^{n\,\delta\,\Lambda} + e^{2\,\Lambda\,\delta}\, \sigma_n(x)\, .
\end{align*}
We easily deduce by induction that
\begin{align*}
\sigma_n(x)
	\leq	
		C_{\delta,x}\,\omega^{n-1}\,\frac{\omega^n-1}{\omega-1}
\end{align*}
where $\omega=e^{\Lambda\,\delta}$.
Then
\begin{align*}
  \var_{\delta_{x}}\Bigl(
	\crochet{e^{-\Lambda \, n \, \delta}\, \eta_{n \, \delta}, \, \hat v} 
  \Bigr)
  =
  \omega^{-2\,n}\,\sigma_n(x)
  \leq 
		C_{\delta,x}\, \frac{1-\omega^{-n}}{\omega^2-\omega}
\end{align*}
which is bounded if $\omega>1$, that is $\Lambda>0$.
\end{proof}
%----------------------------------------

\medskip

%----------------------------------------
\begin{proof}[Proof of Theorem \ref{th.critere}]
From Lemma \ref{lemme.martingale.bornee},
$\crochet{ e^{-\Lambda \, t}\, \eta_t, \, \hat v}$
is a non-negative martingale under the probability $\PP_{\delta_{x}}$. That gives the existence of $\mathcal{Z}$ such that \eqref{limit.Z} holds.
Moreover for any $t>0$:
\begin{align}
\label{eq.th.critere1}
  \EE_{\delta_{x}}(\crochet{ e^{-\Lambda \, t}\, \eta_t, \, \hat v})  = 
  \EE_{\delta_{x}}(\crochet{  \eta_0, \, \hat v}) = \hat v(x) \,.
\end{align}

\begin{enumerate}
\item \label{pt.proof.surcrit} \textit{Proof of \ref{surcritique}}. If $\Lambda>0$, from Lemma \ref{lemme.martingale.bornee}, $\crochet{ e^{-\Lambda \, t}\, \eta_t, \, \hat v}$ is bounded in $L^2$, then $\crochet{ e^{-\Lambda \, t}\, \eta_t, \, \hat v}$ converges towards $\mathcal{Z}$ in $L^1$ under $\PP_{\delta_{x}}$, i.e. 
$$
	\EE_{\delta_{x}}(\crochet{ e^{-\Lambda \, t}\, \eta_t, \, \hat v}) \xrightarrow[t\to\infty]{} \EE_{\delta_{x}}(\mathcal{Z}). 
$$
From \eqref{eq.th.critere1}, we get $\EE_{\delta_{x}}(\mathcal{Z})=\hat v(x)>0$. As the event $\{\mathcal{Z}>0\}$ is included in the event $\{\text{survival}\}$ then $\PP_{\delta_{x}}(\text{survival})>0$.

\item \label{pt.proof.souscrit} \textit{Proof of \ref{sous-critique} for $\Lambda<0$.} In the case when $\Lambda<0$, we conclude directly from Proposition \ref{majoration.esp.Nt} that the population goes to extinction a.s. 

\item \label{pt.proof.surcrit.M} \textit{Proof that $\M_t^\varepsilon\to 0$ a.s. for $\Lambda=0$.}
Assume that the initial individual has mass $x$ at time $t=0$, and let $\varepsilon>0$ such that $x\in(\varepsilon,\mmax-\varepsilon)$ and 
$\textstyle q\left(\left[ \varepsilon / (\mmax-2\,\varepsilon), \, 1/2\right]\right)>0$. 

\medskip

Let $c>0$ and $T_0\geq 0$ be fixed. There exists $t_0>0$ such that
$$
1>
\gamma
	:=
		\inf_{\varepsilon \leq x \leq \mmax-\varepsilon} 
		\PP_{\delta_{x}}\left(M_{t_0}^\varepsilon \geq c \right)>0\,.
$$		

Indeed, we can construct an event which is included in the previous one and whose probability is uniformly bounded below in $\varepsilon \leq x\leq \mmax-\varepsilon$. 
For example, we can consider the event where any particle descended from this individual divides when its mass belongs to $(\mmax-2\,\varepsilon, \, \mmax-\varepsilon)$ with a division proportion $\alpha \in \left[ \varepsilon / (\mmax-2\,\varepsilon), \, 1/2\right]$ and this during $B$ generations. In particular, each daughter cell has mass in $[\varepsilon,\,\mmax-\varepsilon]$. For $B$ sufficiently large, we have, for some well chosen $t_0$, the previous property.

Then for any $\eta_0$ such that $\text{Supp } \eta_0 \cap [\varepsilon, \mmax-\varepsilon] \neq \emptyset$, we have
\begin{align}
\label{proba.mt0.majoree}
	\PP_{\eta_0}\left(
					M_{t_0}^\varepsilon \leq c
				\right) \leq 1-\gamma <1 \,.
\end{align}

We define the following sequence of stopping times
$$
	\left\{
		\begin{tabular}{l}
		$\tau_0=T_0$\\
		$\tau_k = \inf\left\{t>\tau_{k-1}+t_0, \, 
						M_t^\varepsilon \geq 1 \right\}$ for any $k\geq 1\,.$
		\end{tabular}
	\right.
$$
Then, using recursively the Markov property and \eqref{proba.mt0.majoree},
\begin{align}
\nonumber
\PP\left(
		\forall t \geq T_0, \,
		M_t^\varepsilon \leq c 
		\text{ and }
		\exists s \geq t, \,
		M_s^\varepsilon \geq 1
	\right)
&\leq
	\PP\left(
		\forall k, \,
		\tau_k<\infty \text{ and } M_{\tau_k+t_0}^\varepsilon\leq c 
	\right)
\\
\label{proba.M.minore}
& \leq
	\prod_{k} (1-\gamma) 
	=0\,.
\end{align}
The previous result holds for any $c>0$, then either $\limsup_{t\to\infty} M_t^\varepsilon = \infty$ or $M_t^\varepsilon \to 0$ a.s.\ If $\limsup_{t\to\infty} M_t^\varepsilon = \infty$ then $\limsup_{t\to\infty}\crochet{\eta_t , \hat v} = \infty$, which contradicts the fact that $\mathcal{Z}$ is integrable.
Hence $M_t^\varepsilon \to 0$ a.s.

\item \label{pt.proof.crit.extinc} \textit{Proof that $M_t^\varepsilon \to 0$ implies extinction}. In view of the paragraph after \eqref{ineq.Mt.varepsilon} in the proof of Proposition \ref{majoration.esp.Nt}, for all $t\geq 0$, we have the
inequality
\begin{equation*}
  K^\varepsilon_t\leq\sum_{(i,s)\in\mathbb{N}\times\mathbb{R}_+}\indic_{\{X^i_s=x_0\}}\indic_{\{X^i_s\text{ satisfies }(P)_{s,t}\}},
\end{equation*}
with the notations used in the proof of Proposition \ref{majoration.esp.Nt}.

We introduce the sequence of stopping times $\theta_0=0$,
\begin{equation*}
  \theta_{i+1}:=\inf\{t>\theta_i:\text{there is an individual of mass }x_0\text{ at time }t\},\quad\forall i\geq 1
\end{equation*}
and $I_i\in\{1,\ldots,N_{\theta_i}\}$ the index of the individual with mass $x_0$ at time $\theta_i$. Note that this sequence is a.s.\
well-defined since the probability that two particles have the same mass after each birth event is zero.
Then, we have for all $t\geq 0$
\begin{equation*}
  K^\varepsilon_t\leq\sum_{i\geq 1}\indic_{\{\theta_i\leq t\}}\Gamma_i(t),
\end{equation*}
where the random variable $\Gamma_i(t)$ is defined for all $t\geq 0$ and $i\geq 1$ as
\begin{equation*}
  \Gamma_i(t):=
  \begin{cases}
    \indic_{\{\text{the individual $I_i$ at time $\theta_i$ satisfies }(P)_{\theta_i,t}\}} & \text{if }t\leq \theta_i \\
    1 & \text{if }t>\theta_i.
  \end{cases}
\end{equation*}
Now, it follows from Markov's property that, for all $t\geq 0$, almost surely on the event $\{\theta_i<\infty\}$,
\begin{equation*}
  \mathbb{E}_{\delta_{x}}(\Gamma_i(t+\theta_i)\mid\mathcal{F}_{\theta_i})=\mathbb{E}_{\delta_{x_0}}(\text{the initial individual satisfies }(P)_{0,t})=e^{-Dt}\, .
\end{equation*}
Since the family of random variables $(\Gamma_i(t),t\geq 0)$ is non-increasing, it follows that
\begin{equation*}
  \forall i\geq 1,\quad \lim_{t\rightarrow+\infty}\Gamma_i(t)=0\quad\text{a.s.\ on }\{\theta_i<\infty\}.
\end{equation*}

We proved above that $M^\varepsilon_t\rightarrow 0$ a.s.\ as $t\rightarrow\infty$. Hence, the number of individuals
which had a mass belonging to $(\varepsilon,M-\varepsilon)$ at some time during their life is a.s.\ finite. Therefore, a.s.\ there exists
$i\geq 1$ such that $\theta_i=\infty$ and
\begin{equation*}
  K^\varepsilon_t\leq \sum_{i\geq 1\text{ s.t.\ }\theta_i<\infty}\Gamma_i(t)\rightarrow 0\quad\text{a.s.}
\end{equation*}
as $t\rightarrow+\infty$, since the right hand side is a.s.\ a finite sum of random variables a.s.\ converging to 0.

\medskip

In the same way as in the proof of Proposition \ref{majoration.esp.Nt}, 
reducing $\varepsilon$ if necessary, such that $\taudiv_0 \eqdef \max_{0\leq x\leq \varepsilon} \taudiv(x) < D$, $L_t^\varepsilon$ is dominated by a branching process with birth rate $\taudiv_0$, death rate $D$ and independent inhomogeneous immigration rate $2\,\bar \taudiv\,(K_t^\varepsilon+M_t^\varepsilon)$. As $K_t^\varepsilon+M_t^\varepsilon \to 0$ a.s.\, then $L_t^\varepsilon \to 0$ a.s.

\medskip

Finally, the population goes to extinction a.s.

\bigskip

\item \textit{Proof of \ref{evt.survie}}. Regardless of the sign of $\Lambda$, we have  $\{\mathcal{Z}>0\} \subset \{survival\}$.

\begin{itemize}
\item[\textbullet] If $\Lambda \leq 0$, from Points \ref{pt.proof.souscrit} to \ref{pt.proof.crit.extinc} we have $\PP_{\delta_{x}}(\text{survival})=0$. Hence $\PP_{\delta_{x}}(\mathcal{Z}>0)=0$.

\item[\textbullet] %Equations \eqref{proba.accumulation.0} and \eqref{proba.M.minore} also hold for $\Lambda>0$. 
Equation \eqref{proba.M.minore} also holds for $\Lambda>0$. Then, from Point \ref{pt.proof.crit.extinc} of this proof for any $A>0$ and $T_0\geq0$ 
\begin{align}
\label{proba.M.borne.nulle}
\PP(\forall t\geq T_0, M_t^\varepsilon \leq A \text{ and survival})=0\,,
\end{align}
i.e. 
for any $A>0$, on the event $\{\text{survival}\}$, there exists a.s.\ $T_A(\omega)<\infty$, first time such that $M_{T_A}^\varepsilon \geq A$. We define the following stopping time 
$$\tau_A = T_A \wedge T_{\text{ext}}<\infty\,,$$ 
where $T_{\text{ext}}$ is the extinction time of the population.
Let $\gamma = \sup_{x\in [\varepsilon,\mmax-\varepsilon]} p(x)$. Then
\begin{align*}
\PP_{\delta_x}(\mathcal{Z}>0)
	&=
	\PP_{\delta_x}\left(\left\{\mathcal{Z}>0\right\} \cap \left\{\tau_A=T_A\right\}\right)
	+\PP_{\delta_x}\left(\left\{\mathcal{Z}>0\right\} 
						\cap \left\{\tau_A=T_{\text{ext}}\right\}
					\right)
\\
	&=
		\PP_{\eta_{\tau_A}}(\mathcal{Z}>0)\times\PP_{\delta_x}(\tau_A=T_A)
		+
		0 
\\
	&\geq
		\left(
			1-\gamma^A
		\right)\,
		\PP_{\delta_x}(\tau_A=T_A)\,.
\end{align*}
From \eqref{proba.M.borne.nulle},
$$
	\left\{\text{survival} \right\}
		=
		\left\{ \limsup_{t\to\infty}M_t^\varepsilon>A \right\}
		\subset \left\{T_A < \infty \right\}\,.
$$
From Point \ref{pt.proof.surcrit}, $\gamma<1$, then for any $\delta>0$, there exists $A_\delta$ such that $\gamma^{A_\delta} \leq \delta$, hence
\begin{align*}
\PP_{\delta_x}(\mathcal{Z}>0)
	&\geq
		\left(
			1-\delta
		\right)\,
		\PP_{\delta_x}(\tau_{A_\delta}=T_{A_\delta})
		\geq 
		\left(
			1-\delta
		\right)\,
		\PP_{\delta_x}(\text{survival})\,.
\end{align*}
We conclude that
\begin{align*}
\PP_{\delta_x}(\mathcal{Z}>0)
	&\geq
		\PP_{\delta_x}(\text{survival})
\end{align*}
and as $\{\mathcal{Z}>0\}\subset\{\text{survival}\}$, the conclusion follows.
\qedhere
\end{itemize}
\end{enumerate}
\end{proof}
%----------------------------------------

%%%%%%%%%%%%%%%%%%%%%%%%%%%%%%%%%%%%%%%%%%%%%%%%%%%%%%%%%%%%%%%%%%%%%%
%%%%%%%%%%%%%%%%%%%%%%%%%%%%%%%%%%%%%%%%%%%%%%%%%%%%%%%%%%%%%%%%%%%%%%
\section{Growth-fragmentation-death eigenproblem}
\label{sec.gfd.eigenpbl}
%%%%%%%%%%%%%%%%%%%%%%%%%%%%%%%%%%%%%%%%%%%%%%%%%%%%%%%%%%%%%%%%%%%%%%
%%%%%%%%%%%%%%%%%%%%%%%%%%%%%%%%%%%%%%%%%%%%%%%%%%%%%%%%%%%%%%%%%%%%%%

This section is devoted to the study of the deterministic model. Under particular assumptions which cover some biological cases as bacterial Gompertz growth,  we prove the existence of eigenelements satisfying Assumption \ref{existence.eigenelements}. The method is similar to the one used by \citep{Doumic2007a} which proves existence of eigenelements for more restrictive growth functions. \cite{Doumic2010a} show such result for non-bounded growth function, which also do not cover the bacterial Gompertz growth.

\bigskip

We consider the eigenproblem
\begin{align}
\label{ann.eq.eigenproblem}
\left\{
\begin{array}{l}
		\partial_x (g(x)\,u(x))
		+(\Lambda+D+\taudiv(x))\, u(x)
		=
  		2\int_0^\mmax \frac{\taudiv(z)}{z}\, 
		q\left(\frac{x}{z}\right)\, 
		u(z)\,\dif z\,,
\\
	g(0)\, u(0) = 0\,,
   \qquad
   D+\Lambda >0\,,
   \qquad
   u(x) \geq 0\,,
   \qquad 
   \int_0^\mmax  u(x)\,\dif x = 1\,,
\end{array}
\right.
\end{align}
and the adjoint problem
\begin{align}
\label{ann.eq.eigenproblem.adjoint}
\left\{
\begin{array}{l}
		g(x)\,\partial_x \phi(x)
		-(\Lambda+D+\taudiv(x))\, \phi(x)
		=
  		2\,\taudiv(x)\,
  			\int_0^1 q(\alpha)\,\phi(\alpha\,x)\,\dif \alpha \,,
\\
   \phi(x) \geq 0\,,
   \qquad 
   \int_0^\mmax u(x)\,\phi(x) \,\dif x = 1\,.
\end{array}
\right.
\end{align}

%----------------------------------------------------------
\begin{hypotheses}
\label{hypotheses}
\begin{enumerate}
\item $q(0)=q(1)=0$.
\item $b\in C[0,\mmax]$. 
\item \label{hyp.equicontinue} The family of functions $\left(x \mapsto \frac{b(y)}{y} \, q\left( \frac{x}{y}\right)\right)_y$ is uniformly equicontinuous, i.e. there exists an application $\omega \in C(\RR^+)$ such that $\lim_{x \to 0}\omega(x)=0$ and for any $y, x_1, x_2 \in [0,\mmax]$
$$
	\left|
	\frac{b(y)}{y} \, q\left( \frac{x_1}{y}\right)
	-
	\frac{b(y)}{y} \, q\left( \frac{x_2}{y}\right)
	\right|	
	\leq
	\omega(|x_1-x_2|) \, ,
$$
with the convention that $q(\alpha)=0$ for $\alpha>1$.

\item \label{hyp.g.2} There exists a non-decreasing function $F\in C(\RR^+)$ such that 
$$
	\int_0^1 \frac{1}{F(x)}\,\dif x = \infty
$$
and for any $x,y \in [0,\mmax]$,
$$
	|g(x)-g(y)| \leq F(|x-y|) \,.
$$
\item \label{hyp.int.finie}
$
\int_0^\infty \int_0^\mmax 
		e^{-\int_0^t b(A_s(y))\, \dif s} \,
		\dif y \, \dif t 
		< \infty \, .
$
\end{enumerate}
\end{hypotheses}
%----------------------------------------------------------

\medskip

Under Assumptions \ref{da.hypo.model.reduit}-\ref{hyp.g} and \ref{hypotheses}-\ref{hyp.g.2}, for any initial condition $x\in(0,\mmax)$, the flow $t\mapsto A_t(x)$ does not reach $\mmax$ in finite time and the inverse flow $t\mapsto A^{-1}_t(x)$ solution of
$$
\left\{
	\begin{array}{l}
	\frac{\partial}{\partial_t} A^{-1}_t(x) = -g(A^{-1}_t(x)) \,,
\\
	A^{-1}_0(x) = x
	\end{array}
\right.
$$
does not reach 0 in finite time. The flow is then a $C^1$-diffeomorphism on $(0,\mmax)$ \cite[Th. 6.8.1]{demazure2000a}.

%----------------------------------------------------------
\begin{remark}
\label{remark.hypo}
Assumption \ref{hypotheses}-\ref{hyp.int.finie} is similar to the one of \cite{Doumic2007a}.
Since $b(0)=0$, it implies that the flow $A_t(x)$ moves away from 0 sufficiently fast enough. So it imposes opposite conditions on $g$ that Assumption \ref{hypotheses}-\ref{hyp.g.2}. We will see below that these assumptions are compatible and are satisfied in a large range of biological situations.
\end{remark}
%----------------------------------------------------------

%----------------------------------------------------------
\begin{theorem}
\label{lemme.psi}
Under Assumptions \ref{da.hypo.model.reduit} and \ref{hypotheses}, there exist $\Lambda\geq-D$ and a non-negative function $\Psi \in C[0,\mmax]$ such that $(\Lambda, \Psi)$ is solution of
$$
 \Psi (x)
 	=
 		2\,\int_0^\infty \int_0^\mmax
 		\frac{b(A_t(y))}{A_t(y)}\,q\left(\frac{x}{A_t(y)} \right)\,\Psi(y)\,
 		e^{-\int_0^t (\Lambda + D+b(A_s(y) )\,\dif s}\,
 		\dif y\, \dif t \,.
$$
\end{theorem}
%----------------------------------------------------------

%----------------------------------------------------------
\begin{corollary}
\label{eigenelements}
Under Assumptions \ref{da.hypo.model.reduit} and \ref{hypotheses}, there exists a solution $(\Lambda, u)$ of \eqref{ann.eq.eigenproblem} such that $u\in C^1(0,\mmax)$. In particular $\Lambda+D>0$.
\end{corollary}
%----------------------------------------------------------

%----------------------------------------------------------
\begin{proof}
For any $x\in (0,\mmax)$ we set
\begin{align}
\label{def.u}
	u(x)
		= \frac{1}{g(x)} \,
			\int_0^x 
				\Psi(y)\, 
				e^{-\int_y^x \frac{\Lambda+D+b(z)}{g(z)}\,\dif z} \,
				\dif y \,.
\end{align}
From \eqref{def.u},
\begin{align*}
\partial_x(g(x)\,u(x)) = -(\Lambda+D+b(x))\,u(x) + \Psi(x)\,.
\end{align*}
From the change of variable $t\mapsto z=A_t(y)$ and Fubini's theorem,
\begin{align*}
\Psi(x)
	&=
		2\,\int_0^\mmax\,\int_y^\mmax 
 		\frac{b(z)}{z}\,q\left(\frac{x}{z} \right)\,\Psi(y)\,
 		e^{-\int_y^{z} \frac{\Lambda + D+b(w)}{g(w)}\,\dif w}\,
 		\frac{\dif z}{g(z)}\, \dif y  
\\
	&=
		2\,\int_0^\mmax
 		\frac{b(z)}{z}\,q\left(\frac{x}{z} \right)\,
 		\frac{1}{g(z)}
 		\int_0^z \Psi(y)\,
 		e^{-\int_y^{z} \frac{\Lambda + D+b(w)}{g(w)}\,\dif w}
 		\dif y \,
 		\dif z 
\\
	&=
		2\,\int_0^\mmax
 		\frac{b(z)}{z}\,q\left(\frac{x}{z} \right)\, u(z)\,
 		\dif z  \,.
\end{align*}
Moreover
\begin{align*}
\int_0^\mmax u(x)\,\dif x
	&=
		\int_0^\mmax 
		\frac{1}{g(x)} \,
			\int_0^x 
				\Psi(y)\, 
				e^{-\int_y^x \frac{\Lambda+D+b(z)}{g(z)}\,\dif z} \,
				\dif y\,\dif x
\\
	&\leq
		\norme{\Psi}_\infty \,	
		\int_0^\mmax 
		\frac{1}{g(x)} \,
			\int_0^x 
				e^{-\int_y^x \frac{\Lambda+D+b(z)}{g(z)}\,\dif z} \,
				\dif y\,\dif x
\\
	&\leq
		\norme{\Psi}_\infty \,	
		\int_0^\mmax 
			\int_0^\infty 
				e^{-\int_0^t (\Lambda+D+b(A_u(y)))\,\dif u} \,
				\dif t \, \dif y \, 
%\\
%	& \leq 
%		\norme{\Psi}_\infty \, \mmax \, 
%		\int_0^\infty e^{-(\Lambda+D)\, t} \, \dif t \,,
\end{align*}
which is finite by Assumption \ref{hypotheses}-\ref{hyp.int.finie}.
Finally $\tilde u(x) = \frac{u(x)}{\int_0^\mmax u(x)\,\dif x}$ is solution of \eqref{ann.eq.eigenproblem}.

\medskip
Multiplying \eqref{ann.eq.eigenproblem} by $x$ and integrating over $[0,\mmax]$, we get
$$
	-\int_0^\mmax g(x)\,u(x)\,\dif x + (\Lambda+D)\,\int_0^\mmax x\,u(x)\,\dif x
	+\int_0^\mmax x\, b(x)\,u(x)\,\dif x
	=
	2\,\int_0^\mmax \int_x^\mmax 
		\frac{b(z)}{z}\,q\left(\frac xz\right)\,u(z)\,\dif z\,.
$$
We easily check that
$$
	2\,\int_0^\mmax \int_x^\mmax 
		\frac{b(z)}{z}\,q\left(\frac xz\right)\,u(z)\,\dif z
		=
		\int_0^\mmax x\, b(x)\,u(x)\,\dif x\,.
$$
From Fubini's theorem and as $\int_0^1 \alpha\,q(\alpha)\,\dif \alpha = \frac 12$,
$$
	(\Lambda+D) 
		= \frac{\int_0^\mmax g(x)\,u(x)\,\dif x}{\int_0^\mmax x\,u(x)\,\dif x}\,,
$$
which is positive because $g(x)>0$ for any $x\in (0,\mmax)$ and $u\neq 0$.
\end{proof}
%----------------------------------------------------------

\medskip

%----------------------------------------------------------
\begin{theorem}
\label{th.pb.adjoint}
Under Assumptions \ref{da.hypo.model.reduit} and \ref{hypotheses}, there exists a solution $\phi \in C[0,\mmax] \cap C^1(0,\mmax)$ of \eqref{ann.eq.eigenproblem.adjoint}.
\end{theorem}
%----------------------------------------------------------

%----------------------------------------------------------
\begin{lemma}
\label{appendix.prop.directes}
Under Assumptions \ref{hypotheses}, the following properties hold:
\begin{enumerate}
\item \label{i} There exists $C_{bq}>0$ such that for any $x,y\in [0,\mmax]$:
\begin{align*}
 \frac{b(y)}{y} \, q\left( \frac{x}{y}\right) \leq C_{bq} \, .
\end{align*} 
\item \label{b_sur_x_borne} There exists $C>0$ such that for any $x \in [0,\mmax]$ 
$$\frac{b(x)}{x} \leq C \,.$$
\end{enumerate}
\end{lemma}
%----------------------------------------------------------

%----------------------------------------------------------
\begin{proof}
Point \ref{i} is a direct consequence of Assumption \ref{hypotheses}-\ref{hyp.equicontinue}. In fact, for $x_2=0$, we get
$$
	\frac{b(y)}{y} \, q\left( \frac{x_1}{y}\right)	
	\leq
	\omega(x_1) \leq C_{bq} \, .
$$

For Point \ref{b_sur_x_borne}, by assumption \ref{hypotheses}-\ref{hyp.equicontinue}, $q$ is continuous on $[0,1]$, so it reaches its maximum for $\alpha_{\max} \in [0,1]$. From Assumption \ref{hypotheses}-\ref{hyp.equicontinue}, for $x_2=0$,
\begin{align*}
\norme{q}_\infty \,\frac{b(x)}{x}
	&=
\frac{b(x)}{x}q\left(\frac{\alpha_{\max}\,x}{x}\right)
	\leq
		\omega(\alpha_{\max}\,x)
\end{align*}
As $\lim_{x \to 0}\omega(\alpha_{\max}\,x)=0$, there exists $C>0$ such that for any $x$, $\frac{b(x)}{x}\leq C$.
\end{proof}
%----------------------------------------------------------

\bigskip

%----------------------------------------------------------
\begin{remark}
\begin{enumerate}
\item Assumption \ref{hypotheses}-\ref{hyp.equicontinue} holds, for example, if $0<\mdiv$
and if for any $\alpha_1,\alpha_2 \in [0,1]$, $q(\alpha_1)-q(\alpha_2) \leq C\,|\alpha_1-\alpha_2|^\beta$, $\beta > 0$. In this case,
$$
	\left|
	\frac{\taudiv(y)}{y} \, q\left( \frac{x_1}{y}\right)
	-
	\frac{\taudiv(y)}{y} \, q\left( \frac{x_2}{y}\right)
	\right|	
	\leq
	C\frac{\taudiv(y)}{y^{\beta+1}}\,|x_1-x_2|^\beta 
	\leq
	C\frac{\bar\taudiv}{{\mdiv}^{\beta+1}}\,|x_1-x_2|^\beta\, .
$$

\item Assumptions \ref{da.hypo.model.reduit}-\ref{hyp.g} and \ref{hypotheses}-\ref{hyp.g.2} hold, for example, for a Gompertz function:
$$
	g(x) = a \, \log\left(\mmax/x\right)\,x \, .
$$	
Indeed, let $x<y$, then
\begin{align*}
g(x)-g(y)
	&=
		a \, \left( x\, \left(
							\ln(y)-\ln(x)
						\right)
					+
					(x-y)\,\ln\left(\textstyle \frac{M}{y}\right)
			 \right)\,.
\end{align*}
Since $x(\ln(y)-\ln(x))\leq y-x$ and $\ln\left(\textstyle \frac{M}{y}\right)<\ln\left(\textstyle \frac{M}{y-x}\right)$, Assumption \ref{hypotheses}-\ref{hyp.g.2} holds for
$F(x)\eqdef a\, x + g(x)$.
\end{enumerate}
\end{remark}
%----------------------------------------------------------

\begin{proposition}
\label{prop.hyp.int.finie}
Assumption \ref{hypotheses}-\ref{hyp.int.finie} holds if there exists $0<m < \mmax$ such that $\inf_{x \in [m,\mmax]}b(x)>0$ and there exist $0<\varepsilon<1$ and $a>0$ such that $g(x) \geq a \, x^{1+\varepsilon}$ in a neighborhood of 0.
\end{proposition}

In particular, Assumption \ref{hypotheses}-\ref{hyp.int.finie} holds for a Gompertz function:
$$g(x) = a\, x \, \log (\mmax / x)$$
or if $g$ is differentiable in 0 such that $g'(0)>0$.

\begin{proof}
Let $b_*\eqdef \inf_{x \in [m,\mmax]}b(x)>0$ and recall the definition \eqref{temps.atteinte.masse} for $t(x,m)$. Then for any $y\leq m$
\begin{align*}
\int_0^\infty 
	e^{-\int_0^t b(A_s(y))\, \dif s} \,\dif t 
& =
	\int_0^{t(y,m)} 
		e^{-\int_0^t b(A_s(y))\, \dif s} \,\dif t
	+ 			
	\int_{t(y,m)}^\infty 
		e^{-\int_0^t b(A_s(y))\, \dif s} \,\dif t
\\
	& \leq
	\int_0^{t(y,m)} 
		e^{-\int_0^t b(A_s(y))\, \dif s} \,\, \dif t
	+ 			
	\frac{1}{b_*}\,.
\end{align*}
Let $0<\delta<m$ be fixed. For any $\delta \leq y\leq m$, we have $t(y,m)\leq t(\delta,m)<\infty$, therefore
\begin{align*}
\int_0^\mmax \int_0^\infty  
		e^{-\int_0^t b(A_s(y))\, \dif s} \,
		\dif t\,\dif y
&=
	\int_0^\delta \int_0^\infty  
		e^{-\int_0^t b(A_s(y))\, \dif s} \,
		\dif t\,\dif y
	+
	\int_\delta^m \int_0^\infty  
		e^{-\int_0^t b(A_s(y))\, \dif s} \,
		\dif t\,\dif y
\\
&\quad
	+\int_m^\mmax \int_0^\infty  
		e^{-\int_0^t b(A_s(y))\, \dif s} \,
		\dif t\,\dif y
\\
&\leq 
	\int_0^\delta \int_0^{t(y,m)} 
		e^{-\int_0^t b(A_s(y))\, \dif s} \,
		\dif t\,\dif y
	+
	(m-\delta)\,t(\delta,m) 
	+\mmax\,\frac{1}{b_*}
\\
&\leq 
	\int_0^\delta t(\delta,m) \,\dif y
	+
	m\,t(\delta,m) 
	+\mmax\,\frac{1}{b_*}\,.
\end{align*}
It is then sufficient to prove that, for a fixed $\delta>0$,
$$
\int_0^\delta t(y,\delta) \,\dif y
		< \infty \,.
$$
Let $x$ be a solution of
$$
	\dot x(t) \geq a\,x(t)^{1+\varepsilon}
$$
then
$$
	x(t) \geq \frac{x(0)}{(1-\varepsilon\,a\,t\,x(0)^\varepsilon)^{1/\varepsilon}}
$$
and
$$
	x(t)\geq \delta 
	\text{ if }
	t \geq \frac{1}{\varepsilon\,a} \, 
		\left(
			\frac{1}{x(0)^\varepsilon}
			-
			\frac{1}{\delta^\varepsilon}
		\right)\,.
$$
We choose $\delta$ such that $g(x)\geq a \, x^{1+\varepsilon}$ for any $x\leq \delta$. Hence
\begin{align*}
\int_0^\delta t(y,\delta)\,\dif y
	&\leq
		\frac{1}{\varepsilon\,a} \, 
		\int_0^\delta
			\left(
			\frac{1}{y^\varepsilon}
			-
			\frac{1}{\delta^\varepsilon}
		\right) \,\dif y
		=
		\frac{1}{\varepsilon\,a} \, 
		\left(
		\int_0^\delta
			\frac{1}{y^\varepsilon} \,\dif y
			-
			\delta^{1-\varepsilon}
		\right)
		<\infty \,. 
\qedhere
\end{align*}
\end{proof}

\medskip

%----------------------------------------------------------
\begin{lemma}
\label{lemme.flot.equicontinu}
Under Assumptions \ref{da.hypo.model.reduit}-\ref{hyp.g} and \ref{hypotheses}-\ref{hyp.g.2}, for any $T>0$, there exists $\omega_T \in C(\RR^+)$ such that $\lim_{x\to 0} \omega_T(x)=0$ and for any $x_1,x_2 \in [0,\mmax]$, $0\leq t\leq T$:
$$
	|A_t(x_1)-A_t(x_2)| \leq \omega_T(|x_1-x_2|)\,.
$$
In the same way, for any $T>0$, there exists $\tilde \omega_T \in C(\RR^+)$ such that $\lim_{x\to 0} \tilde \omega_T(x)=0$ and for any $x_1,x_2 \in [0,\mmax]$, $0\leq t\leq T$:
$$
	|A_t^{-1}(x_1)-A_t^{-1}(x_2)| \leq \tilde \omega_T(|x_1-x_2|)\,,
$$
where $x \mapsto A_t^{-1}(x)$ is the inverse function of $x \mapsto A_t(x)$.
\end{lemma}
%----------------------------------------------------------

%----------------------------------------------------------
\begin{proof}
We set
$$
	H(x) = \int_1^x \frac{1}{F(y)}\,\dif y \, .
$$
The function $H$ is non-decreasing and $H(0)=-\infty$.
Let
$$
	\varphi_t \eqdef H^{-1}(t+H(2\,|x_1-x_2|))\,.
$$
Then $\varphi$ is solution of
$$
	\varphi_t = 2\,|x_1-x_2|+\int_0^t F(\varphi_u)\,\dif u\,.
$$

Let $x_1\neq x_2 \in (0,\mmax)$:
\begin{align*}
|A_t(x_1)-A_t(x_2)|
	&\leq
		|x_1-x_2| + \int_0^t |g(A_u(x_1)-g(A_u(x_2))|\,\dif u
\\
	& \leq
		|x_1-x_2| + \int_0^t F(|A_u(x_1)-A_u(x_2)|)\,\dif u\,.
\end{align*} 
Hence,
\begin{align*}
|A_t(x_1)-A_t(x_2)|
	& \leq
		\varphi_t
	\leq 
	H^{-1}(T+H(2\,|x_1-x_2|)) =: \omega_T(|x_1-x_2|) \,.
\end{align*}
We proceed by the same way with the inverse flow defined by 
\begin{align*}
	A_t^{-1}(x) & = x - \int_0^t g(A_u^{-1}(x))\,\dif u \,. 
\qedhere
\end{align*}
\end{proof}
%----------------------------------------------------------

%%%%%%%%%%%%%%%%%%%%%%%%%%%%%%%%%%%%%%%%%%%%%%%%%%%%%%%%%%%%%%%%%%%%%%
%%%%%%%%%%%%%%%%%%%%%%%%%%%%%%%%%%%%%%%%%%%%%%%%%%%%%%%%%%%%%%%%%%%%%%
\subsection{Proof of Theorem \ref{lemme.psi}}
%%%%%%%%%%%%%%%%%%%%%%%%%%%%%%%%%%%%%%%%%%%%%%%%%%%%%%%%%%%%%%%%%%%%%%
%%%%%%%%%%%%%%%%%%%%%%%%%%%%%%%%%%%%%%%%%%%%%%%%%%%%%%%%%%%%%%%%%%%%%%

In order to prove Theorems \ref{lemme.psi} and \ref{th.pb.adjoint}, we use the same approach as in \cite{Doumic2007a}. In both proofs, we set $D=0$ without loss of generality to simplify notations.

%%%%%%%%%%%%%%%%%%%%%%%%%%%%%%%%%%%%%%%%%%%%%%%%%%%%%%%%%%%%%%%%%%%%%%
\subsubsection{Regularized problem}
%%%%%%%%%%%%%%%%%%%%%%%%%%%%%%%%%%%%%%%%%%%%%%%%%%%%%%%%%%%%%%%%%%%%%%

We set 
$$
	b_\varepsilon(x) = b(x)+\varepsilon \, .
$$	
We consider the Banach space $$E = C[0,\mmax]$$
equipped with the norm $\norme{.}_\infty$.

For any $\lambda\geq0$, $\varepsilon\geq 0$, let $G^\varepsilon_\lambda$ be the operator defined for any $f\in E$ by
\begin{align}
\label{def.G.regul}
 G^\varepsilon_\lambda f(x)
 	=
 		2\,\int_0^\infty \int_0^\mmax
 		\left[
 			\frac{b(A_t(y))}{A_t(y)}\,q\left(\frac{x}{A_t(y)} \right)
 			+\frac{\varepsilon}{\mmax}
 		\right]\,f(y)\,
 		e^{-\int_0^t (\lambda+b_\varepsilon(A_s(y) )\,\dif s}\,
 		\dif y\, \dif t \,.
\end{align}

%----------------------------------------------------------
\begin{lemma}
\label{lemme.hyp.aa}
\begin{enumerate}
\item \label{equibornitude} For any $\lambda \geq 0$, $\varepsilon \geq 0$, $f\in E$, we have $G^\varepsilon_\lambda\, f\in E$ and
\begin{align*}
	\norme{G^\varepsilon_\lambda\, f}_\infty 
	 \leq 2\,\left(C_{bq}+\frac{\varepsilon}{\mmax}\right)\,\norme{f}_\infty \,
	 	\int_0^\infty \int_0^\mmax
 			e^{-\int_0^t b(A_s(y)) \,\dif s}\,
 			\dif y\, \dif t \, .
\end{align*}

\item \label{equicontinuite} 
For any $\lambda \geq 0$, $\varepsilon \geq 0$, $f\in E$, $x_1,x_2 \in [0,\mmax]$: 
\begin{align*}
\left | G^\varepsilon_\lambda f(x_1) -  G^\varepsilon_\lambda f(x_2) \right|
 	&\leq
 		2\, \omega(|x_1-x_2|)\,\norme{f}_\infty \,
 		\int_0^\infty 
 		\int_{0}^\mmax		
 		\,e^{-\int_0^t b(A_s(y))\,\dif s}\,\dif y\, \dif t\, .
\end{align*}
\end{enumerate}
In particular, $G^\varepsilon_\lambda$ is compact on $E$.
\end{lemma}
%----------------------------------------------------------

%----------------------------------------------------------
\begin{proof}
Point \ref{equibornitude} is trivial.
Let $x_1,\,x_2\in[0,\mmax]$, from Assumption \ref{hypotheses}-\ref{hyp.equicontinue},
\begin{align*}
& \left | G^\varepsilon_\lambda f(x_1) -  G^\varepsilon_\lambda f(x_2) \right|
\\
 	& \qquad \leq
 		2\,\int_0^\infty \int_0^\mmax 		
 			\frac{b(A_t(y))}{A_t(y)}\,
 			\left|
 				q\left(\frac{x_1}{A_t(y)} \right)
 				-
 				q\left(\frac{x_2}{A_t(y)} \right)
 			\right|
 			\,f(y)\,
 		e^{-\int_0^t (\lambda + b_\varepsilon(A_s(y) )\,\dif s}\,
 		\dif y\, \dif t
 \\
 	&\qquad \leq
 		2\, \omega(|x_1-x_2|)\,\norme{f}_\infty \,
 		\int_0^\infty 
 		\int_0^\mmax		
 			e^{-\int_0^t b(A_s(y))\,\dif s}\,\dif y\, \dif t \, .
\qedhere
\end{align*}
\end{proof}
%----------------------------------------------------------

\medskip

For any $\varepsilon>0$, $f\in C[0,\mmax]$ such that $f\geq 0$ and $f \neq 0$, we have
\begin{align}
\label{G.positif}
 G^\varepsilon_\lambda f(x)
 	& \geq
 		2\,\frac{\varepsilon}{\mmax} \,
 		\int_0^\infty \int_0^\mmax
 		f(y)\,
 		e^{-\int_0^t (\lambda+b_\varepsilon(A_s(y) )\,\dif s}\,
 		\dif y\, \dif t 
 		>0\,.
\end{align}

Krein-Rutman theorem (see for example \cite{perthame2007a}) allows to deduce the following result.
%----------------------------------------------------------
\begin{corollary}[Eigenelements]
\label{elements.propre}
For any $\lambda\geq 0$ and $\varepsilon>0$, there exist a unique eigenvalue $\mu^\varepsilon_\lambda>0$ and a unique eigenvector $N^\varepsilon_\lambda\in C[0,\mmax]$ such that
$$
	G^\varepsilon_\lambda \, N^\varepsilon_\lambda(x) 
	= \mu^\varepsilon_\lambda \, N^\varepsilon_\lambda(x) \,,
	\qquad N^\varepsilon_\lambda(x)>0 \,,
	\qquad \max_{x\in[0,\mmax]} N^\varepsilon_\lambda(x)=1 \,.
$$
\end{corollary}
%----------------------------------------------------------

%----------------------------------------------------------
\begin{lemma}[Fixed point]
\label{point.fixe}
For any $\varepsilon>0$, there exists $\Lambda_\varepsilon>0$ such that $\mu^\varepsilon_{\Lambda_\varepsilon}=1$, i.e. $\Psi_\varepsilon=N^\varepsilon_{\Lambda_\varepsilon}$ is a fixed point of $G^\varepsilon_{\Lambda_\varepsilon}$:
$$
	G^\varepsilon_{\Lambda_\varepsilon} \Psi_\varepsilon(x) = \Psi_\varepsilon(x)\,.
$$
\end{lemma}
%----------------------------------------------------------

%----------------------------------------------------------
\begin{proof}
Let $\varepsilon>0$ be fixed. First, we prove that $\lambda \mapsto \mu_\lambda^\varepsilon$ is continuous. From Lemma \ref{lemme.hyp.aa}, the family $(G^\varepsilon_\lambda\,N^\varepsilon_{\lambda})_{\lambda\geq 0}$ is compact. Let $\lambda \geq 0$ and let $(\lambda_n)_n$ be a non-negative sequence with limit $\lambda$. There exists a subsequence $\alpha_n$ such that 
$\mu^\varepsilon_{\lambda_{\alpha_n}}\,N^\varepsilon_{\lambda_{\alpha_n}} 
= G^\varepsilon_{\lambda_{\alpha_n}}\,N^\varepsilon_{\lambda_{\alpha_n}} \underset{n\to \infty}{\longrightarrow} \Psi\in E$ with $\Psi \geq 0$.
From \eqref{G.positif}
\begin{align*}
\mu^\varepsilon_{\lambda_{\alpha_n}} \,N^\varepsilon_{\lambda_{\alpha_n}}(x)
	&=
	G^\varepsilon_{\lambda_{\alpha_n}}\,N^\varepsilon_{\lambda_{\alpha_n}}(x)
\\
	& \geq
	2\,\frac{\varepsilon}{\mmax} \,
 		\int_0^\infty 
 		e^{-(\lambda+\bar b +\varepsilon)\,t}\, \dif t \,
 		\int_0^\mmax N^\varepsilon_{\lambda_{\alpha_n}}(y)\,\dif y\,.
\end{align*}
Integrating the previous inequality, we deduce that
\begin{align*}
\mu^\varepsilon_{\lambda_{\alpha_n}} 
	& \geq
	2\,\varepsilon \,
 		\int_0^\infty 
 		e^{-(\lambda+\bar b +\varepsilon)\,t}\, \dif t \,.
\end{align*}
Then $\Psi\neq 0$.

From Lemma \ref{lemme.hyp.aa}
$$
	G^\varepsilon_{\lambda}\,\Big(\mu^\varepsilon_{\lambda_{\alpha_n}}
	\,N^\varepsilon_{\lambda_{\alpha_n}}\Big) 
	\underset{n\to\infty}{\longrightarrow} G^\varepsilon_\lambda \Psi\,.
$$
Moreover from \eqref{def.G.regul}
\begin{align*}
\norme{\Big(G^\varepsilon_{\lambda_{\alpha_n}}-G^\varepsilon_{\lambda}\Big)\,
\Big(\mu^\varepsilon_{\lambda_{\alpha_n}}
\,N^\varepsilon_{\lambda_{\alpha_n}}\Big)}_\infty
	&\leq
	2\,\norme{\mu^\varepsilon_{\lambda_{\alpha_n}}
			\,N^\varepsilon_{\lambda_{\alpha_n}}}_\infty \,
		\left( C_{bq}\mmax + \varepsilon \right) \,
		\int_0^\infty 
		e^{-\varepsilon\,t}\,
		\left|e^{-\lambda_{\alpha_n}\,t}-e^{-\lambda\,t}\right|\,\dif t\,.
\end{align*}
From Lemma \ref{lemme.hyp.aa}-\ref{equibornitude}, the sequence $(\mu^\varepsilon_{\lambda_{\alpha_n}}\,N^\varepsilon_{\lambda_{\alpha_n}})_n
= (G^\varepsilon_{\lambda_{\alpha_n}}\,N^\varepsilon_{\lambda_{\alpha_n}})_n$ is bounded.
Moreover, since
$
	e^{-a_1}\,\left(1-e^{-(a_2-a_1)}\right)
	\leq
	a_2-a_1
$ for any $0\leq a_1<a_2$, we have
\begin{align*}
\norme{\Big(G^\varepsilon_{\lambda_{\alpha_n}}-G^\varepsilon_{\lambda}\Big)\,
\Big(\mu^\varepsilon_{\lambda_{\alpha_n}}
\,N^\varepsilon_{\lambda_{\alpha_n}}\Big)}_\infty
	&\leq
		2 \,\norme{\mu^\varepsilon_{\lambda_{\alpha_n}}
			\,N^\varepsilon_{\lambda_{\alpha_n}}}_\infty \,
		\left( C_{bq}\mmax + \varepsilon \right) \,
		\left|\lambda_{\alpha_n}-\lambda\right|\,
		\int_0^\infty 
		t\,e^{-\varepsilon\,t}\,\,\dif t
\\
	& \underset{n\to+\infty}{\longrightarrow} 0 \,.
\end{align*}
Hence
\begin{align*}
\mu^\varepsilon_{\lambda_{\alpha_n}}\,
\underbrace{\mu^\varepsilon_{\lambda_{\alpha_n}}\,
N^\varepsilon_{\lambda_{\alpha_n}}}_{\to \Psi}
	&=
	G^\varepsilon_{\lambda_{\alpha_n}}\,\Big(\mu^\varepsilon_{\lambda_{\alpha_n}}
	\,N^\varepsilon_{\lambda_{\alpha_n}}\Big) 
\\
	&=
	G^\varepsilon_{\lambda}\,\Big(\mu^\varepsilon_{\lambda_{\alpha_n}}
	\,N^\varepsilon_{\lambda_{\alpha_n}}\Big) 
	+
	\Big(G^\varepsilon_{\lambda_{\alpha_n}}-G^\varepsilon_{\lambda}\Big)\,
	\Big(\mu^\varepsilon_{\lambda_{\alpha_n}}
	\,N^\varepsilon_{\lambda_{\alpha_n}}\Big)
\\
	& \underset{n\to+\infty}{\longrightarrow}
	 G^\varepsilon_\lambda \, \Psi \, .
\end{align*}
Then $\Psi>0$.

By existence and uniqueness of eigenelements, we then deduce that
$$
\mu^\varepsilon_{\lambda_{\alpha_n}} \to \mu^\varepsilon_\lambda
$$
and then $\lambda \mapsto \mu^\varepsilon_\lambda$ is continuous.

\medskip

From Lemma \ref{appendix.prop.directes}-\ref{i}, for any $f \in E$, 
$$
	G^\varepsilon_\lambda \, f(x) 
		\leq
		2\,\left(C_{bq}+\frac{\varepsilon}{\mmax}\right)\,
		\int_0^\infty 
		e^{-(\lambda+\varepsilon) \,t} \, \dif t
		\int_0^\mmax
 		f(y)\,
 		\dif y 
	=
	2\,\left(C_{bq}+\frac{\varepsilon}{\mmax}\right)\, 
	\int_0^\mmax f(y)\, \dif y  \,
	\frac{1}{\lambda+\varepsilon}\,.
$$
Then we get
\begin{align}
\label{maj.mu}
	\mu^\varepsilon_\lambda\,\int_0^\mmax N^\varepsilon_\lambda(x) \, \dif x
	 = \int_0^\mmax G^\varepsilon_\lambda \, N^\varepsilon_\lambda(x) \, \dif x
	\leq
	2\,\left(C_{bq}\mmax + \varepsilon\right)\, 
	\frac{1}{\lambda+\varepsilon}\,
	\int_0^\mmax N^\varepsilon_\lambda(x) \, \dif x \, .
\end{align}
Hence
\begin{align}
\label{mu.to.0}
\lim_{\lambda \to \infty} \mu^\varepsilon_\lambda = 0 \, .
\end{align}

On the other hand, by integrating \eqref{def.G.regul} with respect to $x$, we have
\begin{align*}
\int_0^\mmax G^\varepsilon_\lambda \, f(x) \, \dif x
	&=
	2\,\int_0^\infty \int_0^\mmax
 		(b(A_t(y))+\varepsilon)\,f(y)\,
 		e^{-\int_0^t (\lambda + b_\varepsilon(A_s(y) )\,\dif s}\,
 		\dif y\, \dif t
\\
	&=
	2\,\int_0^\infty \int_0^\mmax
		e^{-\lambda\, t}\,f(y)\,
 		\partial_t \left(-
 		e^{-\int_0^t b_\varepsilon(A_s(y) \,\dif s}\right)\,
 		\dif y\, \dif t\,.
\end{align*}
By integration by parts,
\begin{align*}
\int_0^\mmax G^\varepsilon_\lambda \, f(x) \, \dif x
	&=
	2\,\int_0^\mmax f(y) \, \dif y
	-	
	2\,\int_0^\mmax f(y) \,
	\int_0^\infty 
		\lambda \,
		e^{-(\lambda+\varepsilon)\, t}\,\,
 		e^{-\int_0^t b(A_s(y) \,\dif s}\,
 		\dif t \,\dif y\, .
\end{align*}
Therefore
\begin{align*}
\mu_\lambda \, \int_0^\mmax N^\varepsilon_\lambda(x)\,\dif x
	= 
	2\, \int_0^\mmax N^\varepsilon_\lambda(x)\,\dif x 
	- 2\,\lambda\,\int_0^\mmax  N^\varepsilon_\lambda(y) \,
	\int_0^\infty 
		e^{-(\lambda+\varepsilon)\, t}\,\,
 		e^{-\int_0^t b(A_s(y) \,\dif s}\,
 		\dif t \, \dif y   
\end{align*}
and 
\begin{align}
\label{mu.equals.2}
\lim_{\lambda \to 0} \mu^\varepsilon_\lambda = 2 \, .
\end{align}

From \eqref{mu.to.0}, \eqref{mu.equals.2} and since the function $\lambda \to \mu^\varepsilon_\lambda$ is continuous, there exists $\Lambda_\varepsilon>0$ such that $\mu^\varepsilon_{\Lambda_\varepsilon} = 1$.
\end{proof}
%----------------------------------------------------------

%%%%%%%%%%%%%%%%%%%%%%%%%%%%%%%%%%%%%%%%%%%%%%%%%%%%%%%%%%%%%%%%%%%%%%
\subsubsection{Proof of Theorem \ref{lemme.psi}}
\label{dem.th.vp}
%%%%%%%%%%%%%%%%%%%%%%%%%%%%%%%%%%%%%%%%%%%%%%%%%%%%%%%%%%%%%%%%%%%%%%

For any $\varepsilon>0$, $(\Lambda_\varepsilon, \Psi_\varepsilon)$ is defined by Lemma \ref{point.fixe}.
From Lemma \ref{lemme.hyp.aa}, the family $(\Psi_\varepsilon)_{0<\varepsilon\leq 1}$ is compact. Moreover, the sequence $\Lambda_\varepsilon$ is bounded. Indeed, it follows from \eqref{maj.mu} that
$$
	0\leq \Lambda_\varepsilon \leq \, 2\,(C_{bq}\,\mmax +1).
$$
Then we can extract a subsequence $(\Lambda_\varepsilon, \Psi_\varepsilon)_\varepsilon$ which converges towards $(\Lambda, \Psi) \in \RR^+ \times C[0,\mmax]$ when $\varepsilon \to 0$.
Moreover,
$$
 \Psi_\varepsilon (x)
 	=
 		2\,\int_0^\infty \int_0^\mmax
 		\left[
 			\frac{b(A_t(y))}{A_t(y)}\,q\left(\frac{x}{A_t(y)} \right)
 			+\frac{\varepsilon}{\mmax}
 		\right]\,\Psi_\varepsilon(y)\,
 		e^{-\int_0^t (\Lambda_\varepsilon+b_\varepsilon(A_s(y) )\,\dif s}\,
 		\dif y\, \dif t \,.
$$
From Assumption \ref{hypotheses}-\ref{hyp.int.finie} and dominated convergence theorem
$$
 \Psi (x)
 	=
 		2\,\int_0^\infty \int_0^\mmax
 			\frac{b(A_t(y))}{A_t(y)}\,q\left(\frac{x}{A_t(y)} \right)
 		\,\Psi(y)\,
 		e^{-\int_0^t (\Lambda+b(A_s(y) )\,\dif s}\,
 		\dif y\, \dif t \,.
$$

%%%%%%%%%%%%%%%%%%%%%%%%%%%%%%%%%%%%%%%%%%%%%%%%%%%%%%%%%%%%%%%%%%%%%%
%%%%%%%%%%%%%%%%%%%%%%%%%%%%%%%%%%%%%%%%%%%%%%%%%%%%%%%%%%%%%%%%%%%%%%
\subsection{Proof of Theorem \ref{th.pb.adjoint}}
\label{sec.proof.th.pb.adjoint}
%%%%%%%%%%%%%%%%%%%%%%%%%%%%%%%%%%%%%%%%%%%%%%%%%%%%%%%%%%%%%%%%%%%%%%
%%%%%%%%%%%%%%%%%%%%%%%%%%%%%%%%%%%%%%%%%%%%%%%%%%%%%%%%%%%%%%%%%%%%%%

We prove that there exists a non-negative function $\phi\in C[0,\mmax]$ solution of
$$
	\phi(x)
		=
			2\, \int_0^\infty 
				e^{-\int_0^t (\Lambda + D+b(A_s(x)))\,\dif s}\,
				b(A_t(x))\, 
				\int_0^{A_t(x)} \frac{1}{A_t(x)}\,q\left( \frac{y}{A_t(x)}\right)\, \phi(y)\,\dif y\,
				\dif t \, .
$$
Then, by the change of variable $t\mapsto z=A_t(x)$, we can prove that $\phi$ is $C^1(0,\mmax)$ and is solution of
$$
g(x)\,\partial_x \phi(x)
		-(\Lambda+D+\taudiv(x))\, \phi(x)
		=
  		2\,\taudiv(x)\,
  			\int_0^1 q(\alpha)\,\phi(\alpha\,x)\,\dif \alpha\,.
$$

%%%%%%%%%%%%%%%%%%%%%%%%%%%%%%%%%%%%%%%%%%%%%%%%%%%%%%%%%%%%%%%%%%%%%%
\subsubsection{Regularized problem}
%%%%%%%%%%%%%%%%%%%%%%%%%%%%%%%%%%%%%%%%%%%%%%%%%%%%%%%%%%%%%%%%%%%%%%

For any $\lambda\geq 0$, $\varepsilon>0$, let $G^{\varepsilon,*}_\lambda$ be the operator defined for $f\in E$ by
$$
	G^{\varepsilon,*}_\lambda \, f(x)
		=
			2\, \int_0^\infty 
				e^{-\int_0^t (\lambda + b_\varepsilon(A_s(x)))\,\dif s}\,
				\int_0^{\mmax} %{A_t(x))} 
					\left[
						\frac{b(A_t(x))}{A_t(x)}\,
						q\left( \frac{y}{A_t(x)}\right)
						+ \frac{\varepsilon}{\mmax}
					\right]						
					\, f(y)\,\dif y\,
				\dif t \,.
$$
Then $G^{\varepsilon,*}_\lambda$ is the adjoint operator of $G^\varepsilon_\lambda$ defined by \eqref{def.G.regul}.

%----------------------------------------------------------
\begin{lemma}
\label{lemme.hyp.aa.adjoint}
\begin{enumerate}
\item For any $\lambda \geq 0$, $\varepsilon \geq 0$, $f \in C[0,\mmax]$, we have
$$
	\norme{G^{\varepsilon,*}_\lambda \, f}_\infty
		\leq 2\,\norme{f}_\infty \, .
$$

\item For any $\lambda_0> 0$, there exists $\tilde \omega_{\lambda_0}\in C(\RR^+)$ such that $\lim_{x \to 0}\tilde \omega_{\lambda_0}(x)=0$ and such that for any $\lambda \geq \lambda_0$, $\varepsilon \geq 0$, $f \in C[0,\mmax]$, $x_1,x_2\in (0,\mmax)$ we have
$$
\left| G^{\varepsilon,*}_\lambda \, f(x_1)
	-G^{\varepsilon,*}_\lambda \, f(x_2) \right|
	\leq
	2\,\norme{f}_\infty\,
	\tilde \omega_{\lambda_0}(|x_1-x_2|) \, .
$$
\end{enumerate}
In particular, $G^{\varepsilon,*}_\lambda$ is compact in E.
\end{lemma}
%----------------------------------------------------------

%----------------------------------------------------------
\begin{proof}
We compute
\begin{align*}
G^{\varepsilon,*}_\lambda \, f(x)
		&\leq
		2\, \norme{f}_\infty \, \int_0^\infty 
				e^{-\int_0^t b_\varepsilon(A_s(x))\,\dif s}\,
				\int_0^{\mmax} %{A_t(x))} 
					\left[
						\frac{b(A_t(x))}{A_t(x)}\,
						q\left( \frac{y}{A_t(x)}\right)
						+ \frac{\varepsilon}{\mmax}
					\right]\,\dif y\,
				\dif t
\\
	&=
			2\, \norme{f}_\infty \,
			\int_0^\infty
				(b(A_t(x))+\varepsilon)\, 
				e^{-\int_0^t b_\varepsilon(A_s(x))\,\dif s}\,		
				\dif t 
	=
		2\, \norme{f}_\infty\,.
\end{align*}

Let $x_1<x_2$. For any $a\geq 0$,
\begin{align*}
& \left|G^{\varepsilon,*}_\lambda \, f(x_1)-G^{\varepsilon,*}_\lambda \, f(x_2)\right|
	\leq
		2\, \norme{f}_\infty \, (\bar b+\varepsilon)\,
			\int_a^\infty 
				e^{-\lambda\,t}\,
				\dif t
\\	
	&\qquad \qquad
		+
		2\, \norme{f}_\infty\, \int_0^a 
				e^{-\int_0^t (\lambda + b_\varepsilon(A_s(x_1)))\,\dif s}\,
\\
	&\qquad \qquad \qquad \qquad
				\times \left|
				\int_0^{\mmax}  
					\left[
						\frac{b(A_t(x_1))}{A_t(x_1)}\,
						q\left( \frac{y}{A_t(x_1)}\right)
						-
						\frac{b(A_t(x_2))}{A_t(x_2)}\,
						q\left( \frac{y}{A_t(x_2)}\right)
					\right]					
					\,\dif y\, \right|
				\dif t
\\
	&\qquad  \qquad
	+	2\, \norme{f}_\infty\,	
	 \int_0^a
	 			e^{-\lambda\,t}\,
	 			\left|
	 			e^{-\int_0^t b_\varepsilon(A_s(x_1))\,\dif s}
	 			-
				e^{-\int_0^t b_\varepsilon(A_s(x_2))\,\dif s}
				\right|\,
\\
	&\qquad \qquad \qquad \qquad
	\times
		\int_0^{\mmax}
					\left[
						\frac{b(A_t(x_2))}{A_t(x_2)}\,
						q\left( \frac{y}{A_t(x_2)}\right)
						+ \frac{\varepsilon}{\mmax}
					\right]						
					\,\dif y\,
				\dif t\,.
\end{align*}
Hence,
\begin{align*}
& \left|G^{\varepsilon,*}_\lambda \, f(x_1)-G^{\varepsilon,*}_\lambda \, f(x_2)\right|
	\leq
		2\, \norme{f}_\infty \, (\bar b+\varepsilon)\,
			\int_a^\infty 
				e^{-\lambda\,t}\,
				\dif t
\\	
	&\qquad \quad
		+
		2\, \norme{f}_\infty\, \int_0^a 
				e^{-\int_0^t (\lambda + b_\varepsilon(A_s(x_1)))\,\dif s}\,
\\
	&\qquad \qquad \qquad \qquad
		\times\Bigg[
			\int_0^{A_t(x_1)}
			\left|
				\frac{b(A_t(x_1))}{A_t(x_1)}\,
				q\left( \frac{y}{A_t(x_1)}\right)
				-
				\frac{b(A_t(x_2))}{A_t(x_2)}\,
				q\left( \frac{y}{A_t(x_2)}\right)
			\right|					
			\,\dif y
\\	
	&\qquad \qquad \qquad \qquad \quad
		+
		\int_{A_t(x_1)}^{A_t(x_2)}
			\frac{b(A_t(x_2))}{A_t(x_2)}\,
				q\left( \frac{y}{A_t(x_2)}\right)				
		\,\dif y	
		\Bigg]\,
		\dif t
\\
	&\quad \qquad
	+	2\, \norme{f}_\infty\,(\bar b+\varepsilon)\,	
	 \int_0^a
	 			e^{-\lambda\,t}\,
	 			\left|
	 			\int_0^t 
	 				\left(
	 					b_\varepsilon(A_s(x_1))
	 					-
						b_\varepsilon(A_s(x_2))
					\right)\,\dif s
				\right|
				\dif t
\end{align*}
From Assumption \ref{hypotheses}-\ref{hyp.equicontinue}
\begin{align*}
& \int_0^{A_t(x_1)}
		\left|
			\frac{b(A_t(x_1))}{A_t(x_1)}\,
				q\left( \frac{y}{A_t(x_1)}\right)
			-
			\frac{b(A_t(x_2))}{A_t(x_2)}\,
				q\left( \frac{y}{A_t(x_2)}\right)
		\right|					
	\,\dif y
\\
 & \qquad \leq
	\int_0^{A_t(x_1)}
			\frac{b(A_t(x_1))}{A_t(x_1)}\,
				\left|
					q\left( \frac{y}{A_t(x_1)}\right)
					-
					q\left( \frac{y}{A_t(x_2)}\right)
				\right|
	\dif y
\\
	&\qquad \quad
	+
	\int_0^{A_t(x_1)}
			\left|			
			\frac{b(A_t(x_1))}{A_t(x_1)}\,
			-
			\frac{b(A_t(x_2))}{A_t(x_2)}\,
			\right|
				q\left( \frac{y}{A_t(x_2)}\right)					
	\,\dif y
\\
 & \qquad \leq
 \int_0^{A_t(x_1)}
			\omega\left(\left| y-y\,\frac{A_t(x_1)}{A_t(x_2)} \right| \right)
	\dif y
	+
		\left|			
			\frac{A_t(x_2)}{A_t(x_1)}\,b(A_t(x_1))\,
			-
			b(A_t(x_2))
		\right|\,.
\end{align*}
From Lemma \ref{appendix.prop.directes}-\ref{b_sur_x_borne}, we then get
\begin{align*}
& \left|G^{\varepsilon,*}_\lambda \, f(x_1)-G^{\varepsilon,*}_\lambda \, f(x_2)\right|
	\leq
		2\, \norme{f}_\infty \, (\bar b+\varepsilon)\,
			\int_a^\infty 
				e^{-\lambda\,t}\,
				\dif t
\\	
	& \qquad \quad
		+
		2\, \norme{f}_\infty\, \int_0^a 
		\Bigg[
		\int_0^{A_t(x_1)}
			\omega\left(\left| \frac{y}{A_t(x_2)}\,\left(A_t(x_2)-A_t(x_1)\right) \right| \right)
			\dif y
\\
	&\qquad \qquad \qquad \qquad \quad
		+\left(C+C_{bq}\right) \,
			\left|A_t(x_2)-A_t(x_1)\right|
			+
			\left|b(A_t(x_1)-b(A_t(x_2)\right|
		\Bigg]\,
		\dif t
\\
	&\quad \qquad
	+	2\, \norme{f}_\infty\,(\bar b+\varepsilon)\,	
	 \int_0^a
	 			\left|
	 			\int_0^t 
	 				\left(
	 					b(A_s(x_1))
	 					-
						b(A_s(x_2))
					\right)\,\dif s
				\right|
				\dif t \, .
\end{align*}

Since $b\in C[0,\mmax]$ and from Lemma \ref{lemme.flot.equicontinu}, the flow $A$ is uniformly equicontinuous on $[0,a]\times [0,\mmax]$, there exists $\omega_a \in C(\RR^+)$ such that $\lim_{x \to 0}\omega_a(x)=0$ and
\begin{align*}
\left|G^{\varepsilon,*}_\lambda \, f(x_1)-G^{\varepsilon,*}_\lambda \, f(x_2)\right|
	&\leq
		2\, \norme{f}_\infty \, C_{\text{st}}\,
		\left[			
			\int_a^\infty 
				e^{-\lambda\,t}\,
				\dif t
			+
			 \omega_a(|x_1-x_2|)
		\right]\,.
\end{align*}

For any $\delta>0$, there exists $a_\delta>0$ such that 
$$
	\int_{a_\delta}^\infty 
				e^{-\lambda\,t}\,
				\dif t \leq \frac{\delta}{2} \, .
$$
Therefore, for any $\delta>0$ we have
\begin{align*}
\left|G^{\varepsilon,*}_\lambda \, f(x_1)-G^{\varepsilon,*}_\lambda \, f(x_2)\right|
	&\leq
		2\, \norme{f}_\infty \, C_{\text{st}}\,
		\left[	
 		\frac{\delta}{2}
 		+ \omega_{a_\delta}(|x_1-x_2|)
 		\right]\, .
\qedhere
\end{align*}
\end{proof}
%----------------------------------------------------------
\medskip

For any $\varepsilon>0$ and $f\in C[0,\mmax]$ such that $f\geq 0$ and $f\neq 0$, we have 
\begin{align*}
G^{\varepsilon,*}_\lambda \, f(x)
		\geq
			2\, \frac{\varepsilon}{\mmax}\, 
			\int_0^\infty 
				e^{-\int_0^t (\lambda + b_\varepsilon(A_s(x)))\,\dif s}\,
				\int_0^{\mmax} 					
					 f(y)\,\dif y\,
				\dif t >0 \,.
\end{align*}

From Krein-Rutman theorem, we deduce the following result. 
%----------------------------------------------------------
\begin{corollary}[Eigenelements]
For any $\lambda> 0$, $\varepsilon>0$, there exist a unique eigenvalue $\tilde\mu^\varepsilon_\lambda>0$ and a unique eigenvector $\phi^\varepsilon_\lambda\in C[0,\mmax]$ such that
$$
	G^{\varepsilon,*}_\lambda \phi^\varepsilon_\lambda (x) 
	= \tilde \mu^\varepsilon_\lambda \, \phi^\varepsilon_\lambda (x) \,,
	\qquad \phi^\varepsilon_\lambda(x)>0 \,,
	\qquad \max_{x\in[0,\mmax]} \phi^\varepsilon_\lambda(x)=1 \,.
$$
\end{corollary}
%----------------------------------------------------------

%----------------------------------------------------------
\begin{lemma}[Fixed point]
\label{point.fixe.adjoint}
For any $\varepsilon>0$, let $\Lambda_\varepsilon>0$ be as defined by Lemma \ref{point.fixe}. Then 
$\tilde\mu^\varepsilon_{\Lambda_\varepsilon}=\mu^\varepsilon_{\Lambda_\varepsilon}=1$ and $\phi_\varepsilon=\phi^\varepsilon_{\Lambda_\varepsilon}$ satisfies
%Pour tout $\varepsilon>0$, il existe $\Lambda_\varepsilon>0$ tel que $\tilde\mu^\varepsilon_{\Lambda_\varepsilon}=1$, c'est-à-dire que $\phi_\varepsilon=\phi^\varepsilon_{\Lambda_\varepsilon}$ est un point fixe de $G^{\varepsilon,*}_{\Lambda_\varepsilon}$ :
$$
	G^{\varepsilon,*}_{\Lambda_\varepsilon} \phi_\varepsilon(x) = \phi_\varepsilon(x)\,.
$$
\end{lemma}

\begin{proof}
For any $\lambda >0$, $\varepsilon>0$, let $N_\lambda^\varepsilon$ be as defined by Corollary \ref{elements.propre}. We have
\begin{align*}
\tilde \mu^\varepsilon_\lambda \,
\int_0^\mmax \phi^\varepsilon_\lambda(x)\,N^\varepsilon_\lambda(x) \, \dif x
	&=
		\int_0^\mmax 
			G^{\varepsilon, *}_\lambda \phi^\varepsilon_\lambda(x)\,
			N^\varepsilon_\lambda(x) \, \dif x
\\
	&=
		\int_0^\mmax 
			\phi^\varepsilon_\lambda(x)\,
			G^\varepsilon_\lambda N^\varepsilon_\lambda(x) \, \dif x
\\
	&=
		\mu^\varepsilon_\lambda \, 
		\int_0^\mmax 
			\phi^\varepsilon_\lambda(x)\,
			N^\varepsilon_\lambda(x) \, \dif x\,.
\end{align*}
Hence
$$
	\tilde \mu^\varepsilon_\lambda=\mu^\varepsilon_\lambda \,.
$$
From Lemma \ref{point.fixe}, there exists $\Lambda_\varepsilon$ such that
$\tilde \mu^\varepsilon_{\Lambda_\varepsilon}=1$.
\end{proof}
%----------------------------------------------------------

%%%%%%%%%%%%%%%%%%%%%%%%%%%%%%%%%%%%%%%%%%%%%%%%%%%%%%%%%%%%%%%%%%%%%%
\subsubsection{End of the proof of the Theorem \ref{th.pb.adjoint}}
%%%%%%%%%%%%%%%%%%%%%%%%%%%%%%%%%%%%%%%%%%%%%%%%%%%%%%%%%%%%%%%%%%%%%%

For any $\varepsilon>0$, $(\Lambda_\varepsilon, \Phi_\varepsilon)$ is defined by Lemma \ref{point.fixe.adjoint}. From Section \ref{dem.th.vp} and Corollary \ref{eigenelements}, we can extract a subsequence $(\Lambda_{\varepsilon_i})_i$ which converges towards $\Lambda>0$ when $\varepsilon_i \to 0$. We can assume that the elements of this subsequence are positive and then there exists a lower bound $\lambda_0>0$ of this subsequence. From Lemma \ref{lemme.hyp.aa.adjoint}, the family $(\phi_{\varepsilon_i})_i$ is compact, we can then extract a subsequence of $(\Lambda_{\varepsilon_i}, \Psi_{\varepsilon_i})_\varepsilon$ which converges towards $(\Lambda, \phi) \in \RR^+ \times C[0,\mmax]$ when $\varepsilon_i \to 0$.
Moreover,
$$
\phi_\varepsilon(x)
		=
			2\, \int_0^\infty 
				e^{-\int_0^t (\Lambda_\varepsilon + b_\varepsilon(A_s(x)))\,\dif s}\,
				\int_0^{\mmax} %{A_t(x))} 
					\left[
						\frac{b(A_t(x))}{A_t(x)}\,
						q\left( \frac{y}{A_t(x)}\right)
						+ \frac{\varepsilon}{\mmax}
					\right]						
					\, \phi_\varepsilon(y)\,\dif y\,
				\dif t \,. 
$$
From Assumption \ref{hypotheses}-\ref{hyp.int.finie}
$$
	\int_0^\mmax e^{-\int_0^t b(A_s(x)))\,\dif s}\, \dif t <\infty
	\quad \text{for almost any } x\in [0,\mmax]\,.
$$
Then, the dominated convergence theorem entails that, for almost any $x\in[0,\mmax]$,
$$
\phi(x)
		=
			2\, \int_0^\infty 
				e^{-\int_0^t (\Lambda + b(A_s(x)))\,\dif s}\,
				\int_0^{A_t(x))}
						\frac{b(A_t(x))}{A_t(x)}\,
						q\left( \frac{y}{A_t(x)}\right)					
					\, \phi(y)\,\dif y\,
				\dif t\,. 
$$
Moreover, $\phi$ is continuous, as well as the function defined by the right term. 
Then the previous inequality is true for any $x\in [0,\mmax]$.
Therefore, $\phi$ is solution of \eqref{ann.eq.eigenproblem.adjoint}.

%%%%%%%%%%%%%%%%%%%%%%%%%%%%%%%%%%%%%%%%%%%%%%%%%%%%%%%%%%%%%%%%%%%%%%
\section{Application for adaptive dynamics to a chemostat model}
\label{sec.chemostat}
%%%%%%%%%%%%%%%%%%%%%%%%%%%%%%%%%%%%%%%%%%%%%%%%%%%%%%%%%%%%%%%%%%%%%%

In a more general context, the growth-fragmentation-death model of Section \ref{subsec.mecha} can describe a mutant population in a variable environment.
For instance, in the study of a chemostat, we can consider the following model structured by trait and mass: each individual is characterized by a phenotypic trait $c\in\C$, where $\C$ is some trait space, say a measurable subset of $\RR^d$, $d \geq 1$, and by its mass $x \in [0,\mmax]$.
We consider the following mechanisms:
\begin{enumerate}
\item \label{da.item.division} \textbf{Division/mutation:} each individual $(c,x)$ divides, at rate $\taudiv(S,c,x)$, into two individuals with masses $\alpha \, x$ and $(1-\alpha)\,x$, where the proportion $\alpha$ is distributed according to a kernel $Q(c, \dif \alpha)=q(c, \alpha)\,\dif \alpha$ and $S$ is the substrate concentration in the chemostat.
\begin{itemize}
\item With probability $\mut \in [0,1]$, the daughter cell $\alpha\,x$ is a mutant bacterium, with trait $c+h \in \C$, where $h$ is distributed according to a kernel $\kappa(c, h)\,\dif h$ and the daughter cell $(1-\alpha)\,x$ has the same trait $c$ as its mother.

\item With probability $1-\mut$, the two daughter cells have the same trait $c$ as the mother cell. 
\begin{center}
\includegraphics[width=6cm]{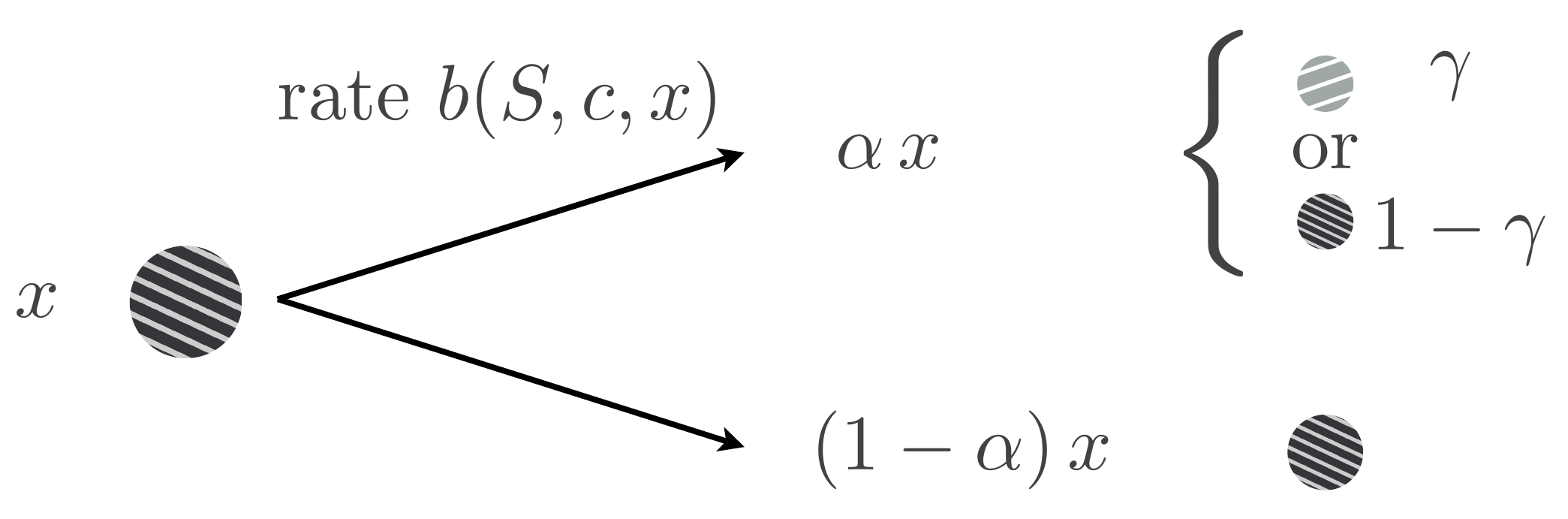}
\end{center}
\end{itemize}

\item \label{da.item.soutirage} \textbf{Washout:} each individual $(c,x)$ is withdrawn from the chemostat at rate $D$, where $D$ is the dilution rate of the chemostat.

\item \label{da.item.croissance} \textbf{Growth:} between the times of division and washout, the mass of individuals grows at speed $g$:
\begin{align*}
  \frac{\dif}{\dif t} X_t^i 
  = 
  g(S_t, C_t^i, X_t^i)\,, \quad i=1,\dots, N_t\,.
\end{align*}

\item \textbf{Substrate dynamics:} the substrate concentration evolves according to the following equation:
\begin{align*}
  \frac{\dif}{\dif t} S_t^n
  =
  D\,(\Sin-S_t) - \frac{k}{n\,V} \, 
						\sum_{i=1}^{N_t} \rhog(S_t,C_t^i,X_t^i)
\end{align*}
where $\left\{(C_t^i,\,X_t^i),\, 1\leq i \leq N_t \right\}$ is the set of traits and masses of the individuals in the population at time $t$, $\Sin$ is the substrate input concentration, $k$ is a stoichiometric coefficient, $n\,V$ is the volume of the vessel and $n$ is a parameter scaling the volume of the vessel.
\end{enumerate}
See \cite{fritschThesis} for more details on the stochastic process.

\medskip

Without mutation, i.e. for $\mut=0$, the previous model was studied by \cite{campillo2014d}. The authors have proved that in the limit $n\to\infty$, the individual-based model associated to the previous mechanisms converges to the following integro-differential model:
\begin{align} 
\label{eq.eid.substrat}
	&
	\frac{\dif}{\dif t} S_t  = 
	D\,(\Sin-S_t)-\frac kV \int_0^{\mmax} g(S_t,x)\,
	\densiteIDE_t(x)\,\dif x\,,
\\
\nonumber
	&
	\frac{\partial}{\partial t} \densiteIDE_t(x)
	+\frac{\partial}{\partial x} \bigl( g(S_t,x)\,\densiteIDE_t(x)\bigr)
	+ \bigl(\taudiv(S_t, x)+D \bigr)\,\densiteIDE_t(x)
\\
\label{eq.eid.pop}
  	&\qquad\qquad\qquad\qquad\qquad\qquad\qquad
	=  
	2\,\int_x^{\mmax}
		\frac{\taudiv(S_t,z)}{z} \, 
		q\left(\frac{x}{z} \right) \,
		\densiteIDE_t(z)\,\dif z \,,
\end{align}
where $r_t$ represents the mass density of the population at time $t$. 

\bigskip

The dynamics of the population trait can be described by the \textit{trait substitution sequence} process whose principle is the following. We assume that the initial population is large and monomorphic and that mutations are rare, so that the initial population reaches and stays in a neighborhood of this stationnary state $(S^*_{c_0}, r^*_{c_0})$ before a mutation occurs.

Under the previous assumptions, just after the first mutation time, the number of mutant individuals, with trait $c$, is negligible with respect to the number of individuals with trait $c_0$, called resident population, as long the mutant population remains small.
The effect of the mutant population on the stationary state $(S^*_{c_0}, r^*_{c_0})$ is then negligible (see \cite{Metz1996a,geritz1998a}).

The mutant population can then be approached, just after the mutation time, by the process \eqref{da.def.proc.S.nu} with
\begin{align*} 
   \taudiv(x)      &\eqdef \taudiv(S^*_{c_0},c,x)\,,\\
   q(\alpha) &\eqdef q(c,\alpha)\,,\\
   g(x)      &\eqdef g(S^*_{c_0},c,x)\,.
\end{align*}

Once the mutant population invades, the two traits are in competition. 
In general, the coexistence of both traits $c_0$ and $c$ is not possible, because of the competitive exclusion principle. In fact, \citet{Hsu1977a} have proved a competitive exclusion principle for $g(S,x)=\mu(S)\,x$ (see also \citet{smith1995a}).

If the competitive exclusion principle is satisfied, then two cases are possible at the time of the mutation:
\begin{itemize}
\item \textbf{1rst case :} The mutant population goes to extinction.

\item \textbf{2nd case :} The mutant population invade the resident one. In this case the resident population goes to extinction and the substrate/mutant population pair reaches a neighborhood of its new stationary state. The mutant population then becomes the resident population for the next mutation.
\end{itemize}

\bigskip

The convergence towards a unique stationary state was proved for non-structured chemostat model (see for example \citet{smith1995a}). The rare mutations assumption is related to the initial population size, which is assumed to be large. 
We refer to \citet{champagnat2006b} and \citet{TranChi2009a} for the precise dependence of the mutation probability $\gamma$ on the parameter $n$.

The fact that the initial population stays in a neighborhood of this stationary state before the mutation occurs and that the impact of the mutant population under the initial population is negligible may be justified by large deviations estimates, as in a simplest age-structured context by \citet{vietchitran2008a} and \cite{TranChi2009a}.
We leave this delicate issue for further work.

\medskip

Just after the mutation time, the mutant population is small and subject to strong randomness. Hence, a deterministic model is irrelevant and it is better to modelize it by the stochastic model described in Section \ref{subsec.IBM}. However, the study of the deterministic model described in Section \ref{subsec.edp} brings information on the possibility of invasion of the mutant population, as we have seen in Theorem \ref{th.critere}.

%%%%%%%%%%%%%%%%%%%%%%%%%%%%%%%%%%%%%%%%%%%%%%%%%%%%%%%%%%%%%%%%%%%%%%
%%%%%%%%%%%%%%%%%%%%%%%%%%%%%%%%%%%%%%%%%%%%%%%%%%%%%%%%%%%%%%%%%%%%%%
\section*{Acknowledgements}
%%%%%%%%%%%%%%%%%%%%%%%%%%%%%%%%%%%%%%%%%%%%%%%%%%%%%%%%%%%%%%%%%%%%%%
%%%%%%%%%%%%%%%%%%%%%%%%%%%%%%%%%%%%%%%%%%%%%%%%%%%%%%%%%%%%%%%%%%%%%%
The work of Nicolas Champagnat was partially funded by project MANEGE ‘Mod\`eles Al\'eatoires en \'Ecologie, G\'en\'etique et \'Evolution’ 09-BLAN-0215 of ANR (French national research agency).
The work of Coralie Fritsch was partially supported by the Meta-omics of Microbial Ecosystems (MEM) metaprogram of INRA.

\end{document}